\documentclass{siamart0516}
\usepackage[margin=1in]{geometry}
\usepackage{graphicx} 

\newcommand{\bsub}{\begin{subequations}}
\newcommand{\esub}{\end{subequations}$\!$}

\newcommand{\IE}{\mathbb{E}}
\newcommand{\IV}{\mathbb{V}}
\newcommand{\Sn}{\mathbb{S}^{n-1}}
\newcommand{\Sone}{\mathbb{S}^1}

\newcommand{\al}[1]{\textcolor{black}{#1}}

\usepackage[sort&compress,numbers]{natbib}
\usepackage{graphicx,booktabs}
\usepackage{amsmath}
\usepackage{amssymb,subfigure}
\usepackage{comment}
\usepackage{amsfonts}
\usepackage{pgfplots}
\usepackage{color}
\usepackage[normalem]{ulem}
\definecolor{sectioncolor}{rgb}
{0.9, 0.3, 0.8}
\definecolor{hillencolor}{rgb}{0.50, 0.03, 0.46}
\definecolor{paintercolor}{rgb}{0.1, 0.7, 0.2}
\definecolor{mypink}{rgb}{0.9, 0.3, 0.6}
\definecolor{lightgray}{rgb}{0.1, 0.1, 0.1}
\newcommand{\RR}{\mathbb{R}}
\newcommand{\ep}{\varepsilon}
\newcommand{\ff}{\varphi}

\usepackage{titlesec}
\pgfplotsset{compat=1.18}

\begin{document}

\title{Mean First Passage Times for Transport Equations}
\date{\today}
\author{Thomas Hillen\thanks{Department of Mathematics, University of Alberta, Canada. {\tt thillen@ualberta.ca}}\and
Maria R. D'Orsogna\thanks{Department of Mathematics, California State University at Northridge, Los Angeles, CA, 91330, USA. {\tt dorsogna@csun.edu}}\and
Jacob C. Mantooth\thanks{Deptartment of Applied and Computational Mathematics and Statistics, University of Notre Dame, Notre Dame, IN, 46656, USA. {\tt jmantoot@nd.edu}}
\and Alan E. Lindsay\thanks{Department of Applied and Computational Mathematics and Statistics, University of Notre Dame, Notre Dame, IN, 46656, USA. {\tt a.lindsay@nd.edu}}
}

\maketitle

\begin{abstract}
Many transport processes in ecology, physics and biochemistry can be 
described by the average time to first find a site or exit a region, starting from 
an initial position. Typical mathematical treatments are based on  
formulations that allow for various diffusive forms and geometries but where only 
initial and final positions are taken into account. Here, we develop a 
general theory for the mean first passage time (MFPT) for velocity jump processes. 
For random walkers, both 
position and velocity are tracked and the resulting Fokker-Planck equation takes the form of a kinetic transport equation. Starting from the forward and backward 
formulations we derive a general elliptic integro-PDE for the MFPT of a random walker starting at a given location with a given velocity. We focus on two scenarios that are relevant to biological modelling; the diffusive case and the anisotropic case. For the anisotropic case we also perform a parabolic scaling, leading to a well known anisotropic MFPT equation. 
To illustrate the results we consider a two-dimensional circular domain under radial symmetry, where the MFPT equations can be solved explicitly. Furthermore, we consider the MFPT of a random walker in an ecological habitat that is perturbed by linear features, such as wolf movement in a forest habitat that is crossed by seismic lines.

\end{abstract}

\label{firstpage}



\pagestyle{myheadings}
\markboth{T.~Hillen, M.R.~D'Orsogna, J.~C.~Mantooth, A.E.~Lindsay}{Mean first passage
times for transport equations}

\section{Introduction}
A key issue in the study of species movement is determining the time 
necessary for individuals to find a target, such as a food source, 
shelter, or mate, for the first time \cite{Venu2015,Stepien2020,PaHiTurtle}. In 
ecology, the time required for an animal to first reach (or exit from) a region 
can be recorded via satellite tracking and the data used to understand how the 
environment, road layout, seasonality, climate and the presence of other species affect 
habitat selection, foraging and migration patterns. This information may be useful in 
ecosystem management \cite{FAUCHALD2003, MCKENZIE2009, LECORRE2008, 
LECORRE2014}. Within cellular biology, identifying the time to first reach a 
specific site is also of primary interest, since first arrivals may trigger 
irreversible on-site biochemical transformations that indicate the onset of 
disease, the initiation of repair, or the completion of extracellular signaling 
\cite{Fok2009, BERNOFF2023, Lawley2023, Miles2020, ABNH2023}. Examples include the first time a
T-cell finds an antigen presenting cell during immune response \cite{Morgan2023}, the first time a fibroblast reaches a 
wound to initiate healing \cite{ambrosi2004cell}, the first time a base excision repair 
enzyme reaches a lesion on a DNA strand via charge transport 
\cite{Fok2008}. 
The first arrival time to a target is often a proxy for the 
efficiency of the search and a common goal is to expedite, 
hinder or otherwise control the search process itself \cite{ALLEN2010, CHOU2014,BL2018,LBW2017}. For 
example, in receptor-mediated viral entry a virus can enter a cell only after binding to a critical number of surface receptors as 
in the case of HIV, the SARS coronavirus, and hepatitis C 
\cite{GIBBONS2010, LAGACHE2017,BL2018}. The timing of gene expression also depends on the first 
time a specific set of intracellular proteins reach a threshold concentration 
\cite{RIJAL2020, GHUSINGA2017}. 
Related questions arise in chemistry, material and polymer science, 
and pertain to identifying the first assembly time of a cluster starting from individual 
subunits, determining how fast protein, filaments or other molecular aggregates 
reach a predetermined size, how fast bacteriophage and viral capsids are assembled 
\cite{WEISS1967, YVINEC2012, FARAN2023, CHAE2024,LEECH2022}.

\begin{figure}
    \centering
    \includegraphics[width= 0.5\textwidth]{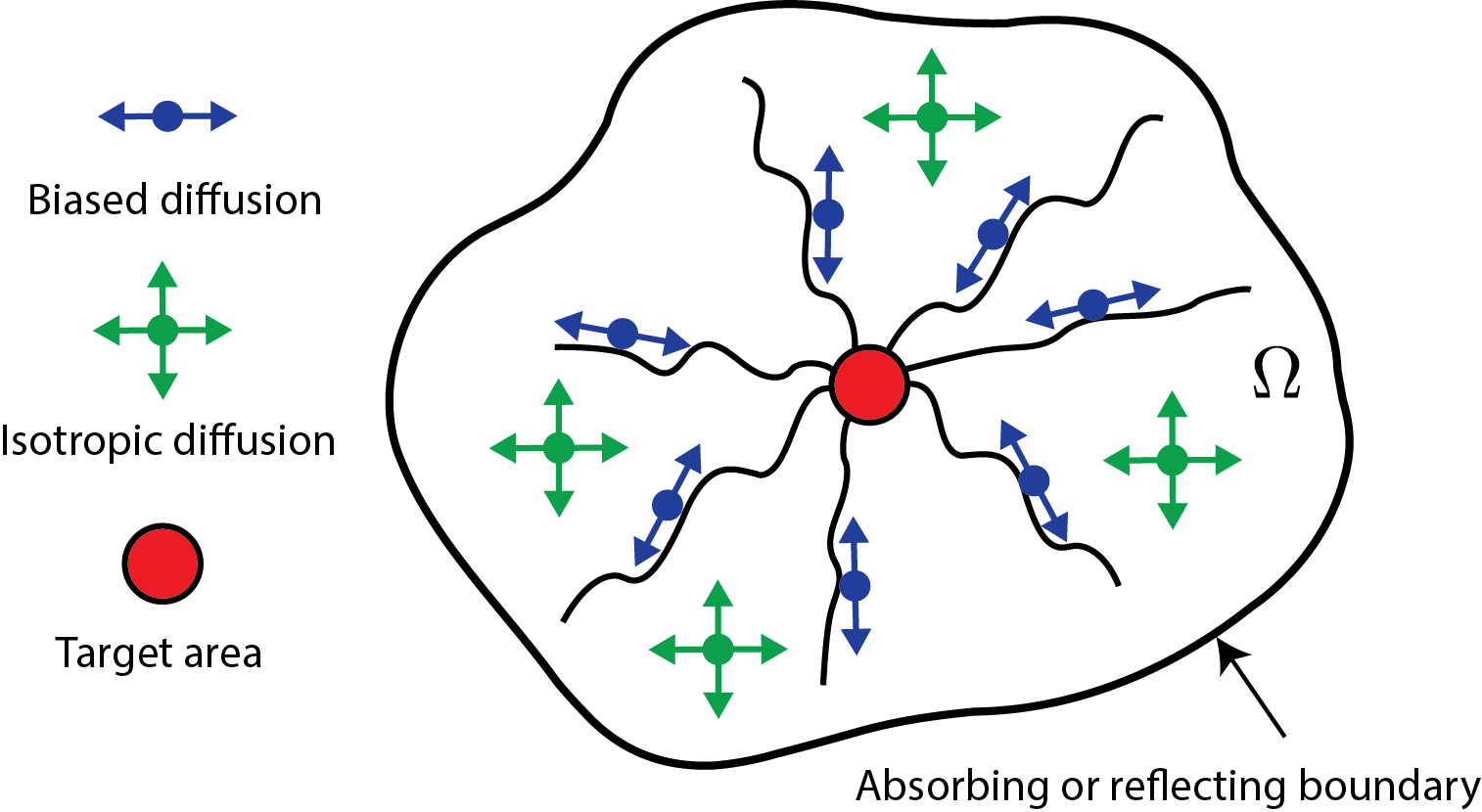}
    \caption{Schematic of anisotropic transport in a bounded planar region $\Omega\in\mathbb{R}^2$ in the presence of one-dimensional substructures. Diffusion is one-dimensional along the linear features and two-dimensional away from them. This schematic can represent, for example, organelle transport within a two-dimensional cell, grown on a flat surface, whereby particles transition from performing linear random walks along microtubules and/or other filaments emanating from the nucleus, to undergoing two-dimensional planar diffusion in the cytoplasm. Similarly, it can represent animal motion in a given environment with roads or other one-dimensional features that direct the animal's movement. Due to the stochastic nature of the transport process, each individual trajectory starting at a given position $x \in \Omega$ and with a given velocity $v$ will reach the target area, shown in red, at a specific time.  Our goal is to determine the mean first arrival time to the target. We can model this transport scenario using a kinetic transport equation with turning kernel given by the bimodal von-Mises distribution in a circular geometry and assuming that motion is biased along the radial direction.}
\end{figure}

First arrival times distributions can be quite broad, especially if outliers or rare 
events are possible. For some state space trajectories finding the target may occur 
quickly, for others the search may take longer, and in some cases it may never be 
completed, being halted by degradation or adsorption into trapped states.  A useful 
quantity is the {\it mean first passage time} 
(MFPT), the average of all first arrival times of the underlying stochastic process 
\cite{REDNER2001}. Note that since the MFPT is an average over all completion times, it 
is enough for one trajectory to never reach the target site or threshold for it to 
diverge. In transport phenomena the MFPT depends on dimensionality, 
geometry, the type of motion involved, heterogeneity of the environment; in 
threshold phenomena it depends on the reversibility of events that drive the process, 
cooperativity, and ordering effects \cite{DORSOGNA2005}. 
Multiple theories have been developed to evaluate the MFPT 
in a variety of spatio-temporal scales and geometries, mostly (but not all) in the 
context of diffusion type models \cite{ward2010a, Newby2013, Schuss2014, 
Lindsay2015,Venu2015, GREBENKOV2016, BressloffSchumm2022}. The most common forms of 
movement are Brownian motion, random walks, Levy flights, ballistic motion, 
and their combinations to represent run-and-tumble motility in bacteria for example, 
or to model alternating periods of movement and rest \cite{ANGELANI2014, PADASH2020, 
MUTOTHYA2021}.

For  however a MFPT theory is missing. These 
equations were introduced in biological modeling in the 1980s, 
and have become a powerful tool in the study of ecological and cell movement
\cite{alt80, HillenOthmer, hilmesenchymal, perthamebook, hillen2013transport, 
PainterButterflies, PaHiTurtle}. 
They are particularly useful when movement characteristics of single
individuals can be measured such as the velocities, turning rates and
directional preferences of migrating cells or animals. 
Recent advances in microscopy and communication 
systems have facilitated the merging of theoretical results with actual 
data, as it is now possible to observe molecular or cellular motion at high
resolution and to 
follow in detail the movement of animals over large distances or underwater. 
It is therefore of great interest to develop a MFPT theory for kinetic equations as first 
passage computations may serve as a bridge between theoretic estimates that depend on 
microscopic features of the transport process and the average timescales observed 
experimentally or in the field. 

Here, we are interested in animal or cellular species that orient their 
movement along directional features of the environment. We derive expressions 
for the mean first arrival time to a target as a function of the initial 
position and velocity. Broadly speaking, since directional motion is 
anisotropic, the mean first arrival time will depend on the microscopic 
features of the transport process, such as turning rates, initial speed and 
direction of movement. For the most general case we are not able to find a 
closed-form equation for the MFPT. However, for diffusive and anisotropic 
transport we are able to write integral equations of the first kind for the 
MFPT starting from an initial position and velocity. In special cases, for 
example for anisotropic, symmetric motion under the bimodal von-Mises distribution, this 
general expression
leads to a differential equation for the MFPT that can be solved analytically
and that can be readily applied to glioma movement \cite{PainterGlioma, 
Kumar2021,SwanHillen,hillen2013transport,engwer2014,Conte2023}, sea turtle 
orientation \cite{PaHiTurtle} and hill topping of butterflies \cite{PainterButterflies}.  
Our main results are contained in \eqref{generalMFPT} and \eqref{anisotropicMFPT}; the analytically solvable variant of \eqref{anisotropicMFPT} for the special case of the bimodal von-Mises distribution is shown in \eqref{anisotropicMFPT3}. 

The rest of this paper is organized as follows. In Section \ref{s:kinetic} we introduce kinetic transport equations for biological modelling, and we
categorize them as of {\it Boltzmann} type, {\it diffusive} type and {\it anisotropic} type. We consider the forward and backward formulations, define the survival probability, and write its dynamics using the backward equation. We thus derive the MFPT equation for the expected exit time $\Theta(x,v)$ of a random walker starting at $(x,v)$ for the diffusive type in Section \ref{s:diffusiveMFPT} and for the anisotropic case in Section \ref{s:anisotropicMFPT}. The main results are given in formulas \eqref{generalMFPT} and \eqref{anisotropicMFPT}, respectively. The parabolic scaling in Section \ref{s:para} allows us to consider the limit in cases of macroscopic spatio-temporal scales. In this scaling we derive an anisotropic MFPT for the 
mean exit time $\Theta(x,v) = T(x)$ which no longer depends on the initial velocity as shown in \eqref{anisotropicMFPT3}. In Section \ref{s:applications} we apply our results to various cases. For anisotropic movement in a radially symmetric circular domain we evaluate the MFPT to reach given boundaries in the disk and in the annulus. A second application considers oriented movement in a habitat that is criss-crossed by linear features. This scenario has been used to explain wolf movement in presence of seismic lines in a forest landscape in Northern Canada 
by \cite{MCKENZIE2009}. Here we show how the analysis carried out by \cite{MCKENZIE2009}
fits within our more general framework. We close with a conclusion Section \ref{s:conclusion} in which we relate our results to existing methods and open the door to interesting forthcoming problems. 

\section{Transport Equations}\label{s:kinetic}
%
\hfill Transport equations describe the time evolution of a particle density \\ $p(x,v,t)$ at time $t\geq 0$, location $x\in\Omega\subset\RR^n$ and velocity $v\in V\subset \RR^n$. 
For simplicity, we 
take 
 $V=[\sigma_1, \sigma_2]\times \Sn$, where the speeds $\sigma_1, \sigma_2$ obey 
$0 \leq \sigma_1 \leq \sigma_2 < \infty$.  
The time evolution of $p(x,v,t)$ is described by the {\em {forward transport equation}}
\begin{equation}\label{kinetic1}
p_t(x,v,t) + v\cdot\nabla_x p(x,v,t) = {\cal L} p(x,v,t)\,,
\end{equation}
where the subscript $t$ denotes the partial time derivative, 
the subscript $x$ denotes the spatial derivative, 
and ${\cal L}$ 
\al{is the {\it turning integral operator}}
that describes the specific directional changes of the particles. A typical form for ${\cal L}$ is 
\begin{equation}\label{integral1}
{\cal L} \varphi(x,v,t) = -\mu(x) \varphi(x,v,t) +\mu(x) \int_V K(x,v,v') \varphi(x,v',t) dv' \, .
\end{equation}
The first term on the right hand side of \eqref{integral1} is 
the rate at which particles at position $x$ change their 
velocity from $v$ to any other velocity $v' \in V$; the second 
term is the rate at which particles at location $x$ switch into 
velocity $v$ from any other velocity $v'$,
The quantity $\mu(x) $ is the {\em {turning
rate}} and its inverse is typically identified as the {\em 
{mean run time}} at position $x$. 
Although general formulations that include spatially dependent 
turning rates $\mu(x)$ arise in many applications \cite{BressloffLawley_2017,BressloffLawley_2017b,BressloffLawley_2017c,Wu2018,BressloffLawleyMurphy_2019}, for simplicity we assume henceforth that
$\mu(x) = \mu$ is spatially uniform.
The {\em {turning kernel}}  $K(x,v,v')$ denotes the probability density of switching velocity from $v'$ to $v$, given that a turn occurs at location $x$. The properties of $K(x,v,v')$ are key to the analysis presented in this work.
The minimal assumptions on $K$ are 
\begin{equation}\label{assumptionsT}
K(x,v,v')\geq 0, \qquad K(x,\cdot ,\cdot )\in L^2(V\times V),
 \qquad \int_VK(x,v,v') dv = 1.
\end{equation}
The second condition implies that the integral operator with kernel $K$ is a compact Hilbert-Schmidt operator in $L^2(V)$ \cite{HillenFA}, and the last assumption ensures that during velocity changes no particle is lost. 
Among the various possible forms of $K(x,v,v')$ are those that allow to 
represent collisions between particles, periods of straight motion alternating with abrupt directional changes (run-and-tumble models)  and general jump velocity processes where a particle's velocity undergoes a series of discrete changes 
\cite{HillenOthmer, othmer-hillen-02, perthamebook, carrillo2010, chaplain2011}. 

Kinetic transport equations, such as (\ref{kinetic1}) have been used in many applications, from physics to biology to social sciences. 
They have a long history in the study of the
thermodynamics of gases which include the effects of collisions between particles \cite{CIP}. In general
formulations of the Boltzmann equation the turning process is described by a non-linear interaction kernel and the set of velocities $V$ 
is unbounded. A linear kernel such as \eqref{integral1} arises as linearization of the collision kernel at a Maxwellian equilibrium distribution \cite{CIP}. An important difference with respect to biological applications is that the Boltzmann equation conserves mass, momentum and energy. {\textit{i.e.}} the kernel of the operator ${\cal L}$ is five dimensional. In biological applications there is only one conserved quantity, mass, and only in cases where there is no particle birth or death. \al{Our focus is on biological applications, and in this context it has proven useful to distinguish 
between the {\it isotropic diffusive} and the {\it anisotropic} sub-classes of transport equations.}

\al{The {\it isotropic diffusive} transport equation is characterized by a turning integral operator ${\cal L}$ whose null space consists of functions that are constant in velocity. This means that at equilibrium there are no preferred directions and that over 
long time-scales the dynamics becomes diffusion-like.  Full theories of the isotropic diffusive transport equations have been developed and applied
to biological processes such as chemotaxis \cite{HillenOthmer,othmer-hillen-02}. 
We derive the MFPT equation for isotropic diffusive transport in Section \ref{s:diffusiveMFPT}.}

\al{The {\it anisotropic} transport equation is typically employed to describe the dynamics of particles that have strong orientation guidance from the underlying environment. Mathematically this occurs via a non-trivial one-dimensional kernel of the turning integral operator ${\cal L}$. The MFPT equation for anisotropic transport is discussed in Section \ref{s:diffusiveMFPT}.}

More specifically, for anisotropic transport 
it is assumed that the turning kernel $K$ does not depend on the incoming
velocity $v'$, i.e. 
\begin{equation}\label{Tq}
K(x,v,v') = q(x,v),
\end{equation}
where $q\geq 0$. This assumption may seem restrictive, since many species would show some form of persistence in movement direction and it is entirely possible to keep the dependence on $v'$  in anisotropic movement as well. However, it has been shown in many applications that the simplifying assumption \eqref{Tq}  is extremely useful in the modelling process and it can be justified in many cases,
for example, when particles change direction using a well-defined underlying network structure.
The anisotropic framework was developed in \cite{hilmesenchymal} and extended in \cite{PainterECM}. To build $q(x,v)$
we assume that a distribution of preferred directions $\tilde q(x,\theta)$ is given for every spatial position $x$, where $\theta\in \Sn$ is a unit vector. We impose
\[ \tilde q (x,\theta) \geq 0, \qquad \mbox{and} \qquad  \int_{\Sn}\tilde q(x,\theta)d\theta =1 \] and assume
$q(x,v)$ is proportional to $ \tilde q(x,\hat v)$, where $\hat v =v/\|v\|$ denotes the corresponding unit vector. 
Since $\tilde q$ is a probability distribution on $\Sn$, and $q$ a probability distribution on the set of velocities $V$, we need to 
normalize appropriately leading to
\begin{equation}
\label{norm}
q(x,v) = \frac{\tilde q(x,\hat v)}{\omega},\quad \mbox{ with }\quad    
\omega \equiv \int_V \tilde q(x,\hat v) dv =  \left\{\begin{array}{ll} 
\frac{1}{n}(\sigma_2^n - \sigma_1^n) & \mbox{for } \sigma_1<\sigma_2;\\[4pt]
\sigma^{n-1} & \mbox{for } \sigma_1=\sigma_2=\sigma.\end{array}\right.
\end{equation}
%
%
We refer to the rescaled quantity $q(x,v)$ as the directional distribution of the underlying environment, 
although the true directional distribution is $\tilde q(x, \hat v)$. For this choice of turning kernel, equation \eqref{integral1} simplifies to
\[ {\cal L} \varphi(v) = \mu(x) (q(x,v)\hat \varphi - \varphi(v)), \qquad \mbox{with} \quad 
\hat \varphi \equiv \int_V \varphi(v') dv'.\]

\noindent
To summarize, the anisotropic transport equation has the form 
\begin{equation}\label{kinetic2}
p_t + v\cdot \nabla p = \mu(x) (q(x,v) \hat p - p) \qquad \mbox{with} \quad 
\hat p \equiv \int_V p(t, x, v') dv'.
\end{equation}

For the following analysis it is useful to consider two statistical quantities, 
the expectation $\IE_q(x)$ and the variance-covariance matrix $\IV_q(x)$
of $q(x,v)$ on $V$: 
\begin{equation}
\label{EV}
 \IE_q(x) = \int_V v q(x,v)dv,\qquad \IV_q(x) =\int_V(v-\IE_q(x))( v-\IE_q(x) )^T q(x,v) dv.
 \end{equation}
We now discuss a few cases of directional distributions $q(x,v)$ by assuming constant particle speed $\sigma$, {\it i.e.} $V= \sigma \Sn$. The extension to $V=[\sigma_1,\sigma_2]\times \Sn$ is straightforward, but it introduces some extra parameters 
that need to be carried through the integrals. The different cases are  \\

\noindent {\bf \al{Uniform distribution:}} The simplest case arises if there is no directional bias at all. This, of course, is also the most fundamental type of motion in diffusive transport. We write
\[ \tilde q(x,\theta) = \frac{1}{|\Sn|}, \qquad q(x,v) = \frac{1}{\omega|\Sn|}. \]
\noindent
leading to
\[ \IE_q(x) = 0, \qquad \IV_q(x) =\frac{\sigma^2}{\omega n}{\mathbb{I}},\]
where $\mathbb{I}$ denotes the identity matrix. 
\al{Note that this case falls under the isotropic diffusive case as well.} \\

 \noindent {\bf \al{Strict alignment:}} This case is most useful when the underlying medium favors motion along a given
direction $\gamma(x)\in\Sn$, as for example along a fiber, tissue, microtubule or track. One can write
\begin{equation*}
\tilde q(x,\theta) = \delta(\gamma(x)-\theta), \qquad q(x,v) = \frac{1}{\omega}\delta(\gamma(x)-\hat v), 
\end{equation*}
so that 
\[ \IE_q(x) = \sigma \gamma(x),\qquad \IV_q(x) = 0.\]
\noindent{\bf \al{Von-Mises distribution:}} The von-Mises distribution is the analog of the normal distribution on the unit sphere $\Sn$. Strictly speaking, it is called von-Mises distribution only on $\Sone$. In higher dimensions it is called the Fisher distribution. In two-dimensions the von-Mises distribution is \cite{vonMises} 
\begin{equation}\label{vonMises}
\tilde q(x, \theta) = \frac{1}{2\pi I_0(k(x)) } e^{k(x) \gamma(x) \cdot \theta},
\end{equation}
where $k(x)\geq 0$ is a measure of concentration of the distribution, 
$\gamma(x)\in\Sone$ is a given preferred direction, and $I_0(k(x))$ is the modified Bessel function of first kind of order $0$. 
As $k(x)\to 0$ the von-Mises distribution becomes uniform (as in 2.3.1)
whereas for $k(x)\to\infty$ it becomes singular 
(as in 2.3.2).
The expectation and variance of the two-dimensional von-Mises distribution are \cite{vonMises}
\bsub
\begin{eqnarray}
 \IE_q(x)  &=& \sigma \frac {I_1(k(x))}{I_0(k(x))} \gamma(x),\label{vm2drift} \\
  \IV_q(x) &=& \sigma^2\left[\frac{1}{2}\left(1 -\frac{I_2(k(x))}{I_0(k(x))} \right) \mathbb{I}_2
+ \left(\frac{I_2(k(x))}{I_0(k(x))} 
-\left(\frac{I_1(k(x))}{I_0(k(x))}\right)^2 \right) \gamma(x) \gamma(x)^T\right] \,,\label{vm2diffusion} 
\end{eqnarray}
\esub
where $I_0, I_1, I_2$ denote the modified Bessel functions of first kind of order 0,1 and 2, respectively. \\

\noindent{\bf \al{Bimodal von-Mises distribution:}} In many cases $q$ is symmetric, for example
when guided motion along a two-dimensional fiber or a track has the same probability along each direction. In this case
$q(x,-v) = q(x, v)$ and the bimodal von-Mises distribution can be written as the average of two single direction terms
\begin{equation}\label{bimodalvonMises}
\tilde q(x, \theta) = \frac{1}{4\pi I_0(k(x)) } \Bigl(e^{k(x) \gamma(x) \cdot \theta} + e^{-k(x) \gamma(x) \cdot \theta}\Bigr),
\end{equation}
with expectation and variance
 \bsub\label{difftensor}
 \begin{eqnarray}
\label{difftensor_a} \IE_q (x) &=& 0\,, \\
 \label{difftensor_b}\IV_q (x) &=&\sigma^2\left[\frac{1}{2}\left(1 -\frac{I_2(k(x))}{I_0(k(x))} \right) \mathbb{I}_2
 + \frac{I_2(k(x))}{I_0(k(x))} \gamma(x) \gamma(x)^T \right]\,.
 \end{eqnarray}
 \esub

Note that instead of using a dyadic product $\gamma\gamma^T$ some authors prefer to use a tensor product $\gamma\otimes\gamma$. 
We use the two interchangeably in this paper. 

\subsection{Backward Transport Equation}

The Kolmogorov backward transport equation is the foundation from which the mean first passage time equation is derived.  For clarity we distinguish initial location and velocity $(x,v)$ from terminal location and velocity $(\bar x, \bar v)$. 
Thus $p(\bar x,\bar v,t |x, v)$, 
denotes the probability density of a random walker starting at location $x\in \Omega\subset \RR^n$ with initial velocity $v\in V$ to be at location $\bar x$ with velocity $\bar v\in V$
at time $t\geq 0$. 
Let us rewrite the forward equation
\begin{equation}\label{tra-forward}
p_t + v\cdot \nabla_{\bar x} p =  -\mu\, p +\mu \int_V K(\bar x, \bar v,\bar v') p(\al{ \bar x}, \bar v', t |x, v)d\bar v',
\end{equation}
%
%
where the gradient subscript allows to distinguish between the spatial derivative with respect to the current location ($\nabla_{\bar x}$)
and the derivative with respect to the starting location ($\nabla$ without index). 
The model \eqref{tra-forward} can be seen as a forward equation of the stochastic process of a velocity jump process \cite{ODA}.  To derive the corresponding Kolmogorov backward equation, we consider the formal adjoint. We write \eqref{tra-forward}
as
\[ p_t (\bar x,\bar v,t | x, v) = {\cal A} p (\bar x,\bar v, t| x, v). \]
The operator $\cal A$  acts on functions $\varphi(\bar x, \bar v)$ of the variables $(\bar x,\bar v)$ in $\Omega\times V$ and is defined as
\[ {\cal A} \varphi (\bar x,\bar v ) = - \bar v\cdot \nabla_{\bar x} \varphi (\bar x,\bar v) 
 -\mu\, \varphi (\bar x,\bar v) +\mu \int_V K(\bar x, \bar v,\bar v') \varphi(\bar x,\bar v')d\bar v'.\]
\al{To proceed, we utilize the Chapman-Kolmogorov 
equation \cite{VanKampen2007} under the assumption of time translational invariance. 
Since we assumed that the turning kernel $K(x,v,v')$ and the turning rate $\mu$ are time independent, the assumption of time translational invariance is satisfied. We thus write}
\begin{equation}
p(\bar x,\bar v, t| x, v) = \int_\Omega \int_V p(\bar x,\bar v, \tau | x', v') p(x', v', t - \tau | x, v) dv' dx',
\end{equation}
for any $\tau \geq 0$. 
Taking the time derivative on both sides we find
\begin{eqnarray*}
p_t(\bar x,\bar v,t | x, v) &=& \int_\Omega \int_V p(\bar x,\bar v, \tau | x_1, v_1) {\cal A} p(x',v', t - \tau | x, v) dv' dx' \\
&=& \int_V \int_\Omega  {\cal A^*} p(\bar x, \bar v, \tau | x', v') 
p(x', v', t - \tau | x, v) dx' dv',
\end{eqnarray*}
where the operator $\cal A^*$ is the adjoint of $\cal A$ and acts on \al{functions of} the variables $(x', v')$ in $\Omega\times V$. 
Notice that if we let $\tau \to t$ then $\lim_{\tau \to t} p(x', v', t - \tau | x, v) =  p(x',v',0 | x, v) = \delta(x' - x) 
\delta (v' - v)$ leading to 
\begin{equation}\label{backward}
p_t  (\bar x,\bar v, t| x, v)= {\cal A^*} p (\bar x,\bar v, t | x, v),  
\end{equation}
where the adjoint $\cal A^*$ acts on \al{functions of} the $(x,v)$ variables in $\Omega\times V$. 
This is the Kolmogorov backward equation.  To find $\cal A^*$ explicitly, given $\cal A$ on $\Omega\times V$, we 
consider two functions $\phi(x,v), \psi(x,v)$ 
\al{that satisfy the Dirichlet boundary condition
\begin{equation}\label{eq:BC_Adjoint}
\phi(x,v) = \psi(x,v) = 0, \quad \mbox{on} \quad (x,v) \in \partial \Omega \times V. 
\end{equation}
Equation\,\eqref{eq:BC_Adjoint} implies that particles are absorbed 
on the boundary of $\Omega$. In principle, the exit boundary may be any positive measure subset of $\partial \Omega$, for example, one may want to 
study the exit from a given region of the boundary or stipulate that certain portions are reflecting. For simplicity, we assume the exit boundary coincides with $\partial \Omega$. In the next section, we will define the \lq\lq survival probability\rq\rq\ that a particle remains in $\Omega$ at time $t$ and we will use $\cal A^*$ to study its dynamics. As particles cannot survive at the exit boundary, the survival probability must be zero on $\partial \Omega$. Thus, since 
obeying Dirichlet boundary conditions is a requirement for the survival probability, we impose that $\cal A^*$ also acts on functions that obey \eqref{eq:BC_Adjoint}. In this way, the survival probability is naturally part of the functional domain of $\cal A^*$. We can now calculate}

%
\begin{eqnarray*}
    \int_\Omega\int_V \phi {\cal A} \psi \,dv\, dx 
    &=& 
    -\int_\Omega\int_V \phi(x,v) (v\cdot \nabla_x)\psi (x,v) \, dv\, dx \\
    && +\int_\Omega\int_V \phi(x,v) \left(-\mu \psi (x,v) +\mu \int_V K(x, v,v') \psi (x,v') dv' \right) dv\, dx\\
    &=& \int_V\int_\Omega \psi(x,v) (v\cdot\nabla_x) \phi(x,v) dx\, dv 
    -\mu \int_V\int_\Omega \phi(x,v) \psi(x,v) dx \, dv  \\
    &&+ \mu\int_V\int_\Omega \left(\int_V K(x, v,v') \phi(x,v) dv' \right) \psi(x,v') \, dx\,dv. 
\end{eqnarray*}
The last identity follows from 
integration by parts and from the Dirichlet boundary condition $\phi(x,v) =\psi(x,v) = 0$ for $x \in \partial \Omega$.
If we now denote the adjoint integral kernel as 
\[ K^*(x, v,v') = K(x, v',v),\]
then $\cal A ^*$, the adjoint operator of $\cal A$, is given by
\[ {\cal A}^*\phi(x,v)  =( v\cdot\nabla) \phi (x,v) -\mu \phi (x,v) + \mu \int_V K^*(x, v,v') \phi(x,v') dv'. \]
The Kolmogorov backward equation in explicit form becomes 
\begin{equation}\label{tra-backward}
p_t - v \cdot \nabla p  = 
 - \mu \,p +\mu \int_V K^*(x, v, v') p(\bar x,\bar v,t |x, v' )dv', 
\end{equation}
This equation was previously derived by Stroock \cite{stroock}.
We can write \eqref{tra-backward}
to mirror \eqref{kinetic1}
by introducing the turning operator ${\cal L}^*$
\begin{equation}\label{kinetic2b}
p_t (t,x,v)- v \cdot \nabla p (t,x,v) = {\cal L}^* p (t,x,v),
\end{equation}
where ${\cal L}^* = {\cal A}^* - (v \cdot \nabla) $ is given by
\begin{equation}
\label{integral2b}
{\cal L}^* \varphi(t,x,v) = -\mu \varphi(t,x,v) +\mu \int_V K^*(x,v,v') \varphi(t,x,v') dv'. \,
\end{equation}
It is straightforward to show that ${\cal L}^*$ is the formal adjoint of
${\cal L}$ defined in \eqref{integral1}. 

\subsection{The Survival Probability}
\label{survprob}
To derive an expression for the mean first passage time of the transport process we first introduce
the survival probability $s(x,v,t)$. This quantity is defined as the probability that having started from an initial
position $x \in \Omega$ with velocity $v \in V$
the final position $\bar x$ and velocity $\bar v$ at time $t$
remain in the same domain $\Omega \times V$. 
We thus assume $p(\bar x, \bar v, t=0| x,v) = \delta (\bar x - x) \delta (\bar v - v)$
with $x \in \Omega$ and $v \in V$ and define
\begin{equation}\label{survive}
s(x, v, t ) = \int_{V} \int_{\Omega}  p(\bar x, \bar v, t| x,v) \, d \bar x\, d \bar v \qquad \forall (x,v) \in\Omega \times V.  
\end{equation}
In principle the domain for the forward and backward equations can be different,
as one can study the exit time from a subset of the domain on which the forward equation
is valid. Here, for simplicity we assume they are the same and utilize absorbing boundary conditions
on $\Omega$. Thus, we set $s(x,v,t) = 0$ for any $x \in \partial \Omega$ 
to indicate that a random walker starting at the 
boundary of $\Omega$ does not survive at any time. Furthermore, the initial condition $p(\bar x, \bar v, t=0 | x, v) = \delta(\bar x - x) 
\delta (\bar v - v)$ implies that $s(x,v,t=0) = 1$ for all $x \in \Omega$.  
We also define the probability flux $j(x,v,t)$ as
\begin{equation}\label{j}
j(x, v, t) = \int_{V} \int_{\Omega} v\, p(\bar x, \bar v,t |x,v) \, d\bar x  \, d \bar v = v s(x,v,t). 
\end{equation}
We now integrate the backward equation \eqref{tra-backward} over the final position $\bar x \in \Omega$ and 
$\bar v \in V$ to find a transport equation for $s(x,v,t)$, 
\begin{eqnarray}
\label{survive_dyn}
s_t - v \cdot \nabla s = - \mu s + \mu \int_V K^*(x,v,v') s(x,v',t) d v'. 
\end{eqnarray}
Given the above definitions of $s(x,v,t)$ and $j(x,v,t)$ we note that one is most often concerned
with the survival probability, and the mean first passage time, of reaching the boundary 
$\partial \Omega$ regardless of the velocity.  Hence we introduce the marginal
survival probability $S(x,t)$ defined as 
\begin{equation}\label{survive2}
S(x, t ) = 
\frac 1 {|V |} \int_{V}
 s(x,v, t) \, d v\, ,
\end{equation}
where $|V| = \int_{V} d v$. The corresponding flux is
\begin{equation}\label{flux2}
J(x, t ) = 
\frac 1 {|V |} \int_{V}
 j(x,v, t) \, d v.
\end{equation}
Note that $S(x,t)$ remains a probability and $S(x,t=0)=1$. The flux $J(x,t)$ still carries units of a velocity. For later use we collect the initial conditions here:
\begin{equation}\label{IC}
 p(\bar x, \bar v, 0 | x, v) = \delta(\bar x - x) \delta (\bar v - v), \qquad s(x,v,0) = 1, \qquad S(x,0)=1.
 \end{equation}
To study the dynamics of $S$ and $J$ we must specify the properties of the turning kernel $K(x,v,v')$ and of its adjoint $K^*(x,v,v')$; we do so in the two cases below for diffusive and anisotropic transport. 

{\al{
\subsubsection{The mean first passage time}}}
{\al{
The quantity of interest in our analysis is $\Theta(x,v)$,
the expected survival time before leaving $\Omega$ 
given the random walker started at position $x$ with velocity $v$. Thus, through $\Theta(x,v)$ we keep the initial velocity explicit in the definition of the MFPT. As $s(x,v,t)$ denotes the survival probability, then $1-s(x,v,t)$ denotes the probability of exit, which has the density $-s_t(x,v,t)$. We calculate via integration by parts that
\begin{equation}
\label{Thetaxv}
 \Theta(x,v) := \ \al{-} \int_0^\infty t s_t(x,v,t) dt = \int_0^{\infty} s(x,v,t) dt.
\end{equation}}

\al{We can directly integrate the backward equation (\ref{survive_dyn}) over time and impose that particles start at position $x$ with velocity $v$, and that none of them survive forever, 
\[ s(x,v,0) =1, \qquad \lim_{t\to \infty} s(x,v,t) = 0, \]
to obtain an integro-differential equation for the expected survival time $\Theta(x,v)$. We find
\begin{equation}\label{Thetakinetic}
-1 = v\cdot \nabla \Theta(x,v) -\mu \Theta (x,v) +\mu \int_V K^*(x,v,v') \Theta(x,v') dv' = {\cal A^*} \Theta(x,v)
\end{equation}
This equation will become useful later. Here, we continue our analysis of the survival probability $S(x,t)$ and its flux $J(x,t)$ for several special cases.}}

{\al{
\subsubsection{Higher moments}
Here we give a brief description of how to derive recursive expressions for the moments $\Theta_m(x,v)$ of the survival probability $s(x,v,t)$ for $m \geq 1$. For $m \geq 0$, the $m$-moment is defined as
\begin{equation}
\Theta_m(x,v) := - \int_0^{\infty} t^m s_t(x,v,t) dt = m \int_0^{\infty} t^{m-1} s(x,v,t) dt, 
\end{equation}
where the last equality was obtained using integration by parts and the assumption that $s(x,v,t)$ decays fast enough at infinity
\[\lim_{t\to \infty} t^{m-1} s(x,v,t) = 0. \]
Then $\Theta_0(x,v) = 1$ denotes the total probability and the first moment $\Theta_1(x,v) = \Theta(x,v)$ is the mean first passage time as defined in \eqref{Thetaxv}. To find the higher moments for $m \geq 2$, we follow a process similar to that leading to \eqref{Thetakinetic}. First, we multiply the backward equation (\ref{survive_dyn}) by $t^{m-1}$ and then integrate
over $t\in(0,\infty)$. Using the same initial conditions as above, 
we find that
\begin{equation}
\label{Thetakinetic_moment}
- m \, \Theta_{m-1}(x,v) = v\cdot \nabla \Theta_m(x,v) -\mu \Theta_m (x,v) +\mu \int_V K^*(x,v,v') \Theta_m(x,v') dv' = {\cal A^*}\Theta_{m}(x,v).
\end{equation}
Upon setting $m=1$ in \eqref{Thetakinetic_moment} we recover the expression for the expected survival time $\Theta(x,v)$ shown in \eqref{Thetakinetic}. 
}}

\section{MFPT for \al{Isotropic} Diffusive Transport}\label{s:diffusiveMFPT} 
We now write the equation for the mean first passage time under the isotropic diffusive transport
regime as discussed above. In addition to \eqref{assumptionsT} we assume symmetry and particle conservation of the turning kernel $K$
\begin{equation}\label{diffusiveK}
K(x,v,v')=K(x,v,-v'), \qquad \mbox{and}\qquad \int_V K(x,v,v') dv' =1.
\end{equation}
For the adjoint operator this implies 
\begin{equation}\label{assumptionT2} 
K^*(x,v,v')=K^*(x,-v,v'), \qquad \mbox{and}\qquad \int_V K^*(x,v,v') dv =1.
\end{equation}
We use assumption \eqref{assumptionT2} and integrate equation \eqref{survive_dyn} for the survival probability $s(x,v,t)$ over velocity space to obtain a conservation law
\begin{eqnarray}\label{cons1}
S_t - \nabla \cdot J = 0. 
\end{eqnarray}
We also multiply the backward equation \eqref{survive_dyn} by $v$,  integrate over velocity space, and divide by the volume $|V|$ to obtain
\begin{eqnarray}\label{cons2}
J_t - \frac 1 {|V|} \int_V v(v\cdot \nabla) s \, dv = -\mu J +\frac{\mu}{|V|}\int_V\int_V v\, K^*(x,v,v') s(x,v',t) \,dv'\, dv .
\end{eqnarray}
Since assumption \eqref{assumptionT2} on the symmetry of $K$ implies that
\[ \int_V vK^*(x,v,v') dv =0, \]
the last integral vanishes. Finally, using tensor notation for a dyadic product of two vectors we  write
\begin{equation}\label{fluxlaw1}
J_t - \frac{1}{|V|}\nabla \cdot \int_V v\otimes v \,s(x,v,t) \, dv = -\mu J.
\end{equation}
Equations \eqref{cons1} and \eqref{fluxlaw1} do not 
 not lead to a closed system of equations for $S$ and $J$. We 
will discuss the closure problem  in more detail later. 
For now, we identify the complement of $S(x,t)$, the function $F(x,t) = 1-S(x,t)$, as the exit probability from 
$\Omega$. Its density is given by
\[ f(x,t) =
F_t(x,t) = -S_t(x,t).\]
The expected survival time  $T(x)$ before exiting $\Omega$ given the random walker started
at the initial position $x \in \Omega$ can be now al{computed through integration by parts as}
\[ T(x) \equiv \int_0^\infty t f(x,t) dt = \int_0^\infty S(x,t) dt. \]
We refer to $T(x)$ as the mean first passage time.
Integrating \eqref{cons1} for the dynamics of $S$, and \eqref{cons2} for the dynamics
of $J$, both over time, and using the initial conditions \eqref{IC}, we obtain the following  
\bsub
\begin{eqnarray}\label{NJsystem2}
\label{NJsystem_a2} -1 &=& \nabla \cdot \int J dt   \\
\label{NJsystem_b2} 0 &=& \frac 1  {|V|} \nabla \cdot \int_V \int_0 ^{\infty}v\otimes v\, s(x,v,t) \, dt \, dv - \mu \int J dt. 
\end{eqnarray}
\esub
\al{In the above calculation, we have assumed that $p(\bar x, \bar v, t \to \infty | x,v) = 0$ since all particles eventually exit the domain and are ``captured" by the boundary due to the Dirichlet boundary condition}. In addition, we have applied that $\int_V v dv = 0$ due to the symmetry of the $V$ domain. We can now take the divergence of \eqref{NJsystem_b2} and substitute the resulting expression in 
\eqref{NJsystem_a2} to obtain
\begin{equation}\label{diffusiveMFPT}
-1 = \frac{1}{\mu |V|} \nabla \otimes \nabla: \int_V \int_0^\infty  v\otimes v \,s(x,v,t) \, dt \,dv.
\end{equation}
Here, we have used the colon notation $:$ to denote the tensor convolution of two vectors $a$ and $b$ as 
\[ a\otimes a : b\otimes b = \sum_{i,j=1}^n a_i a_j b_i b_j .\]
We like to obtain an equation for the exit time $\Theta(s,v)$, defined in (\ref{Thetaxv}). 
Integrating \eqref{Thetaxv} over $V$ we obtain 
\[ T(x) = \frac 1 {|V|} \int_V \Theta(x,v) dv.\]
Using equation \eqref{diffusiveMFPT} we find  the MFPT equation for the case of \al{isotropic} diffusive transport
\begin{equation}\label{generalMFPT} 
-1 = \frac{1}{\mu |V|} \nabla\otimes\nabla: \int_V v\otimes v\, \Theta(x,v) dv. 
\end{equation}

\subsection{Example:} 
We assume constant speed so that $V=\sigma {\mathbb S}^{n-1}$
and rotational symmetry in the initial velocity dependence of $s(x,v,t) =s(x,t) $ so that 
\[ \int_{V} v\otimes v\, s(x,v, t) dv = \frac{\sigma^2}{n} {\mathbb I} \int_V s(x,t) dv , \]
where $\mathbb I$ denotes the identity matrix. Substituting this expression
in \eqref{generalMFPT} \al{and noting that $\nabla \otimes \nabla : \mathbb I = \Delta$}, we find 
\begin{equation} 
\label{ndimdiff}
-1 = \frac{\sigma^2}{n\mu} \Delta T(x).
\end{equation}
This equation has the well known form of a MFPT for a diffusion process in an $n$-dimensional system 
with diffusion constant $D_n=\frac{\sigma^2}{n\mu}$. 
This coefficient also agrees with the parabolic limit of the transport equation, as carried out in several papers, for example \cite{HillenOthmer}. 

\section{MFPT for Anisotropic Transport}\label{s:anisotropicMFPT} 
We now derive the equation for the mean first passage time under anisotropic transport, where directional cues from the environment
bias particle movement along one or more select directions. For simplicity, in addition to \eqref{assumptionsT} we also
assume that $K(x,v,v')  = q(x,v)$ is independent of the incoming velocity $v'$ and that $q(x,v)$ is symmetric in $v$. 
We summarize the properties of $K(x,v,v')$ and $q(x,v)$ below 
\bsub\label{T3assumpt}
\begin{gather}\label{assumptionT3} 
K(x,v,v') = q(x,v), \qquad \mbox{i.e.} \qquad K^*(x,v,v') = q(x,v'),\\[4pt]
\label{assumptionq} 
q\geq 0,\qquad  q\in L^2(V), \qquad \int_V q(x,v) dv =1, \qquad q(x,-v) = q(x,v). 
\end{gather}
\esub
This choice of $K(x,v,v')$ does not satisfy the conditions listed in 
\eqref{assumptionT2}, since 
\[ \int_V K^*(x,v,v') dv = q(x,v') |V|, 
\] 
and the right hand side equals $1$ if and only if $q(x,v')=\frac{1}{|V|}$. 
Hence, the results shown in 
Section \ref{s:diffusiveMFPT} for diffusive transport apply only to the simple
case of a uniform turning kernel but are not valid for a general, non uniform $q(x,v)$. 
To find an expression for the MFPT under anisotropic transport
we start with the backward equation \eqref{survive_dyn} for the survival probability $s(x,v,t)$ adapting it to the case at hand and write
\begin{eqnarray}
\label{survive_dyn2}
s_t - v \cdot \nabla s = - \mu s + \mu \int_V q(x,v') s(x,v',t) d v'. 
\end{eqnarray}
We now multiply equation \eqref{survive_dyn2} by $q(x,v)$ and integrate over 
the velocity domain to obtain
\begin{eqnarray*}
 \int_V q(x,v) s_t(x,v,t) dv - \int_V v q(x,v) \cdot \nabla s(x,v,t) dv && \\
&& \hspace{-5cm} = -\mu \int_V q(x,v) s(x,v,t) dv +\mu \int_V\int_V q(x,v) q(x,v') s(x,v',t) dv' dv =0. 
\end{eqnarray*}
The last identity is derived from the property that $q(x,v)$ integrates to unity in velocity space. To simplify the notation we introduce the following two quantities
\[ n(x,t) = \int_V q(x,v) s(x,v,t) dv \quad \mbox{and} \quad w(x,t) = \int_V v q(x,v)\cdot\nabla s(x,v,t) dv, \]
which leads to
\begin{equation}\label{cons2a}
n_t - w =0.
\end{equation}
We now derive a differential equation for $w$, including dependencies only where explicitly needed 
\begin{eqnarray*}
    w_t &=& \int_V v q \cdot \nabla s_t\, dv
     = \int_V v q \cdot\nabla \left(v\cdot\nabla s -\mu s +\mu \int_V q(x,v') s(x, v',t) dv' \right)dv\\
   &=& \int_V q\,(v\otimes v:\nabla\otimes\nabla) \, s \, dv  
- \mu\int_V v q \cdot \nabla s \, dv + \mu\int_V v q(x,v) dv \cdot \nabla \int_V q(x,v') s(x,v',t)\, dv'.
\end{eqnarray*}
Since $\int_V v q(x,v) dv =0$ by the symmetry assumption in \eqref{assumptionq}, the last term vanishes yielding
\begin{equation}\label{weqa} 
w_t = \int_V q(x,v) (v\otimes v : \nabla\otimes \nabla) s(x,v,t) dv - \mu w.
\end{equation} 
We now integrate \eqref{cons2a} and \eqref{weqa} over time and obtain boundary terms. Since no particle survives forever, {\textit{i.e.}} $s(x,v,t\to\infty)=0$, we can write $n(x,\infty) = w(x,\infty) =0.$ Moreover, the initial condition $s(x,v,0)=1$ yields
$n(x,0) = 1$ and $w(x,0) =0.$ Using these identities we find
\begin{eqnarray*}
    -1 &=& \int_0^\infty w(x,t) dt \\
    0 &=& \int_V \int_0^\infty q(x,v) (v\otimes v:\nabla\otimes\nabla) s(x,v,t) dt \, dv - \mu\int_0^\infty w(x,t) dt, 
\end{eqnarray*}
which combine to yield
\[ -1 = \frac{1}{\mu} \int_V \int_0^\infty q(x,v) (v\otimes v : \nabla\otimes\nabla) s(x,v,t) dv\, dt.\]
We can rewrite this expression by invoking $\Theta(x,v)$, the MFPT where both the initial condition and velocity are specified, to finally obtain \al{the MFPT equation for anisotropic transport:}
\begin{equation}\label{anisotropicMFPT}
-1 = \frac{1}{\mu} \int_V q(x,v) (v\otimes v :\nabla\otimes\nabla) \Theta(x,v) dv.
\end{equation}

\al{
There is an alternative way to derive the above equation (\ref{anisotropicMFPT}) directly from (\ref{Thetakinetic}). The calculations are essentially the same as performed above, but they are carried out in a different order. Still, we find it useful to sketch this derivation here to keep record of the intermediate steps:}

\al{We still set $K(x,v,v') = q(x,v)$ and assume (\ref{assumptionT3}) and (\ref{assumptionq}) to obtain from (\ref{Thetakinetic}) that  
    \begin{equation}\label{help1}
        -1 = v\cdot\nabla \Theta(x,v) - \mu \Theta(x,v) +\mu \int_V q(x,v') \Theta(x,v') dv'. 
    \end{equation}
We multiply (\ref{help1}) by $q(x,v)$ and integrate over $V$ to find 
    \begin{equation}\label{help2}
        -1 = \int_V q(x,v) v\cdot\nabla \Theta(x,v)\; dv.  
    \end{equation}
Then we take the spatial gradient of (\ref{help1}) to get 
    \begin{equation}\label{help3}
    0 = v\cdot \nabla\otimes\nabla \Theta(x,v) -\mu \nabla\Theta(x,v) +\mu \nabla \int_V q(x,v') \Theta(x,v') dv'. 
    \end{equation}
Multiplying (\ref{help3}) by $v q(x,v)$ and integration over $V$ yields:
    \begin{equation}\label{help4} 
        \int_V q(x,v) v\cdot \nabla \Theta(x,v) = \frac{1}{\mu} \int_V q(x,v) \left( v\otimes v:\nabla\otimes\nabla\right) \Theta(x,v) dv,
    \end{equation}
    where we used the symmetry of $q$ such that $\int v q dv =0$. 
Finally, we use (\ref{help4}) in (\ref{help2}) to obtain the main result (\ref{anisotropicMFPT}).}

\medskip

Equation \eqref{anisotropicMFPT} is an interesting anisotropic integro-differential equation for the mean first exit time $\Theta(x,v)$ when starting at $(x,v)$. In this general setting, we cannot derive a simple equation for the mean exit time $T(x)$, since the initial velocity cannot be neglected. Also, we cannot simply integrate \eqref{anisotropicMFPT} over velocity space, since the velocity and derivative terms are mixed. This issue is similar to a moment closure problem, which is well known and well studied in the theory of transport equations \cite{cercignani1988boltzmann,hillen20052}. In order to proceed we must find an approximation for the $q(x,v)$-\al{weighted} second moment of $\Theta(x,v)$. One way to accomplish this is to consider parabolic scaling, which leads to an equation for $T(x)$ alone. We do this in the next section. 

\section{Parabolic Scaling for Anisotropic Transport}\label{s:para}
Parabolic scaling is a well known technique to study transport processes on macroscopic time and space scales. While the scaling of the forward equation is standard, the corresponding scaling for the backward equation is a new result to the best of the authors' knowledge. Here, for consistency, we present both cases, forward and backward.
For simplicity we assume that $K(x,v,v') = q(x,v) = q(x,-v)$ is a symmetric, anisotropic directional distribution that does not depend
on the incoming velocity $v'$ and whose properties are listed
in \eqref{assumptionq}.  We now collect some functional analytical properties of $K(x,v,v') =q(x,v)$.  First, 
for the forward transport equation in \eqref{kinetic1} we write the turning operator $\cal L$ defined via the integral representation 
\eqref{integral1} as a mapping on $L^2(V)$ so that for a test function $\varphi(x,v)$ 
\[ \mathcal{L} \varphi(x,v) = -\mu \varphi (x,v) + \mu q(x,v) \hat{\varphi} (x), \qquad  \hat{\varphi}(x) = \int_V\varphi(x,v') dv'.\]
Similarly, for the backward Kolmogorov equation in \eqref{kinetic2b} the turning operator ${\cal L}^*$ can be written 
as
\[ \mathcal{L}^*\varphi(x,v) = -\mu \varphi(x,v) +\mu \int_V q(x,v') \varphi(x,v') dv'.\]
To perform the parabolic scaling  we need to find the null spaces (kernel) and the pseudo-inverses of the ${\cal L, L^*}$ operators on the complement of 
their null spaces. We collect these properties in the following Lemmas.

\begin{lemma} 
\label{l:kernels}
The kernels of the operators $\mathcal{L}$ and $\mathcal{L}^*$ are the one-dimensional
spaces spanned by the functions $q(x,v)$ and 1, respectively. 
\[ \mbox{ker} \, \mathcal{L} =\langle q(x,v) \rangle, \qquad \mbox{ker} \, \mathcal{L}^* = \langle 1 \rangle.\]
\end{lemma}

\begin{proof}
To find the kernel of $\mathcal{L}$ we seek functions $\phi(x,v)$ that satisfy $\mathcal{L}\phi =0$. This condition is
\[-\mu \phi(x,v) +\mu q(x,v) \hat{\phi} (x) = 0,  \quad \hat{\phi}(x) \equiv \int_V \phi(x,v) dv.\]
To solve the above equations consistently, $\phi(x,v)$ must be proportional to $q(x,v)$. To find the kernel of the adjoint 
${\cal L}^*$  we must instead solve  $\mathcal{L}^*\phi =0$, or
\[ -\mu \phi(x,v) +\mu \int_V q(x,v') \phi(x,v') dv' =0, \]
\al{implying that $\phi(x,v)$ is uniform in $v$, \emph{i.e.} $\phi(x,v) = \phi(x)$}.
\end{proof}
\al{
Note that elements in $\langle 1 \rangle^\perp$ carry no mass since 
\[ 0= (\psi, 1) = \int_V \psi(v) dv \qquad \mbox{for all} \quad \psi\in \langle 1\rangle^\perp.\]}
\begin{lemma} 
\label{l:kernels2} 
For the ${\cal L}$ operator consider the inverse of the generating function 
$q(x,v)$ as a weight in $L^{2}(V)$ and restrict ${\cal L}$ 
to the domain of functions that are $L^2(V)$ when weighted by $q^{-1}$, 
which we denote by $L^2_{q^{-1}} (V)$.
A similar construct for ${\cal L}^*$ using the corresponding generating
function $1$ does not alter the domain of ${\cal L}^*$. Then, the
pseudo-inverse operators acting on 
\[ \mathcal{L}:L^2_{q^{-1}}(V) \mapsto L^2_{q^{-1}}(V), \qquad \mathcal{L}^*: L^2(V) \mapsto L^2(V)\]
are, respectively,  the multiplication operators
\[ \left(\mathcal{L}|_{\langle q\rangle^\perp}\right)^{-1} = -\frac{1}{\mu}, \qquad \left(\mathcal{L}^*|_{\langle 1\rangle^\perp}\right)^{-1} = -\frac{1}{\mu}.\]
\end{lemma}
\begin{proof}
To find ${\cal L}^{-1}$ the pseudo-inverse of $\mathcal{L}$ on $\left(\mbox{ker} \, \mathcal{L} \right)^{\perp}$, the orthogonal complement of $\mbox{ker} \, \mathcal{L}$, 
we consider a given $\phi\in \mbox{ker} \, \mathcal{L}^\perp$ and find $\psi\in \mbox{ker} \, \mathcal{L}^\perp$ that solves $\mathcal{L}\psi = \phi$, so 
that $\psi = {\cal L}^{-1} \phi$. 
The equation $\mathcal{L}\psi = \phi$ reads 
\begin{equation}\label{Linverse}
-\mu \psi(x,v) +\mu q(x,v) \hat {\psi}(x)  = \phi(x,v) .
\end{equation}
The assumption that $\psi\in \langle q\rangle^\perp$ in $L^2_{q^{-1}} (V)$ implies that
\[ 0= (\psi, q)_{q^{-1}} = \int_V \psi(x,v) q(x,v) \frac{dv}{q(x,v)} = \hat{\psi}(x)\]
where $(\psi,q)_{q^{-1}}$ denotes the inner product of $\psi(x,v)$ and $q(x,v)$ in $V$ space using  $q(x,v)^{-1}$ as a weight function. 
Hence \eqref{Linverse} is solved as 
\[ \psi(x,v) = -\frac{1}{\mu} \phi(x,v).\]
Similarly, for $\mathcal{L}^*$ we consider $\phi \in \mbox{ker} \, \mathcal{L^*}^\perp$ 
and find $\psi\in \mbox{ker}\mathcal{L^*}^\perp$ that solves $\mathcal{L^*}\psi = \phi$. We find
\begin{equation}\label{Lstarinverse}
-\mu\psi(x,v) +\mu \int_Vq(x,v')\psi(x,v') dv' = \phi(x,v).
\end{equation}
We integrate this equation over $V$ and use the fact that $\phi,\psi \in \langle 1\rangle^\perp$ to find 
\[ 0 +\mu |V| \int_V q(x,v') \psi(x,v') dv' =0. \]
Hence equation \eqref{Lstarinverse} is solved as 
\[ \psi(x,v) = -\frac{1}{\mu} \phi(x,v).\]
implying that $\left( \mathcal{L}|_{\langle q \rangle^{\perp}}\right)^{-1} = -1/ \mu$. 
\end{proof}

We use these results in the next section where we derive
expressions for the mean first passage time through parabolic scaling
in the case of anisotropic transport where $K(x,v,v') = q(x,v)$ and under
symmetry conditions $q(x,v) = q(x,-v)$.

\subsection{Parabolic Scaling of the Forward Transport Equation}
\label{s:paraforward}

As mentioned above, parabolic scaling of the forward
transport equation is well established in the literature \cite{Bica,hillen2013transport}. 
In the anisotropic case the forward transport equation \eqref{kinetic2} for $p(x,v,t)$ reads 
\begin{equation}
p_t +v\cdot \nabla p = \mathcal{L} p.\end{equation}
We now introduce a small parameter $\ep>0$ and 
rescale macroscopic time and space through the {\it parabolic scaling} variables
$\tau$ and $\xi$ defined as
\begin{equation}\label{scales}
\tau =\ep^2 t, \qquad \xi=\ep x.
\end{equation}
The forward transport equation now becomes 
\begin{equation}\label{qforwardscaled}
\ep^2 p_\tau +\ep v\cdot \nabla_\xi p = \mathcal{L}  p, 
\end{equation}
where $p(\xi, v, \tau)$ is expressed in terms of the rescaled quantities. 
 Note that $\cal L$ contains
no derivatives so the formal expression for ${\cal L} p$ in 
\eqref{qforwardscaled}
is unaffected by the scaling procedure. 
We now expand \eqref{qforwardscaled} in powers of $\ep$ 
\begin{equation}\label{expansion}
    p(\xi,v,\tau) =p_0(\xi,v,\tau) +\ep  p_1(\xi,v,\tau) +\ep^2 p_2(\xi,v,\tau) +\cdots
    \end{equation}
where the functions $p_0, p_1, p_2,\ldots$ are to be determined by 
comparing orders of $\ep$ in  \eqref{qforwardscaled}. Results are as follows

\vspace{0.2cm}

\noindent
\begin{itemize}
\item[$\ep^0:$ ] 
{Collecting the zeroth-order in $\ep$ terms in \eqref{qforwardscaled} leads to  \[ 0 =\mathcal{L} p_0.\]
Hence, $p_0 \in \mbox{ker} \, \mathcal{L}$ by Lemma \ref{l:kernels} and}
\end{itemize}
\begin{equation}
\label{p0}
p_0(\xi,v,\tau) = \hat{p}_{0}(\xi, \tau) q(\xi, v), \qquad \hat {p}_0(\xi, \tau) = \int_V p_0(\xi, v, \tau) dv. 
\end{equation}
\begin{itemize}
\item[$\ep^1:$]
{Collecting the first-order in $\ep$ terms in \eqref{qforwardscaled} leads to
\[ v\cdot\nabla_\xi p_0 =  \mathcal{L} p_1. \]
This equation can be solved if the left hand side is in the orthogonal complement of the kernel of $\mathcal{L}$
in $L^2_{q^{-1}} (V)$. 
A direct calculation yields
\[ (v\cdot\nabla_\xi p_0, q)_{q^{-1}} \equiv \int_Vv\cdot\nabla_\xi \hat {p}_0(\xi,\tau) q(\xi,v) \frac{dv}{q(\xi, v)} =
 \nabla_\xi\cdot  \int_V v q(\xi, v) dv \  \hat{p}_0(\xi, \tau)  =0,\]
where $( v \cdot \nabla_{\xi} p_0 , q )_{q^{-1}}$ denotes the inner product of the two arguments in $V$ space using $q^{-1}$
as a weight function. 
The last identity comes from the symmetry condition $q(x,v) = q(x, -v)$, which implies $q(\xi,v) = q(\xi, -v)$
and thus $\int_V v q(\xi,v) dv =0$. Then,  according to Lemma \ref{l:kernels2}}
\end{itemize}
\begin{equation}
\label{p1}
p_1(\xi, v,\tau) = -\frac{1}{\mu}v\cdot\nabla_\xi\left( \hat {p}_0(\xi,\tau) q(\xi, v)\right) .
\end{equation}
\begin{itemize}
\item[$\ep^2:$ ] 
{Collecting the second-order in $\ep$ terms in \eqref{qforwardscaled} leads to
\[ p_{0\tau} +v\cdot\nabla_\xi p_1 = \mathcal{L} p_2.\]
We integrate this equation over $V$ and use \eqref{p0} and \eqref{p1} to find 
\[ \int_V \hat {p}_{0\tau}(\xi, \tau) q(\xi, v) \, dv +\int_V v\cdot\nabla_\xi \left(-\frac{1}{\mu} v\cdot\nabla_\xi (\hat{p}_0 (\xi, \tau) q(\xi, v)) \right) dv = 0,\]
which we can write using tensor notation as 
\[ \hat {p}_{0\tau} = \frac{1}{\mu}\nabla_\xi\otimes\nabla_\xi:\left[ \hat{p}_0(\tau,\xi) \int_V v\otimes v q(\xi, v) dv \right].\]
Finally, the parabolic limit of the forward transport equation can be written through the introduction of a diffusion tensor $\mathbb{D}(\xi)$ as follows}
\end{itemize}
\begin{equation}
\label{forwardlimit}
\hat p_{0\tau} (\xi, \tau) = \nabla_\xi\otimes\nabla_\xi:\left[\mathbb{D}(\xi) \hat {p}_0(\xi,\tau) \right], \quad \mbox{with} \quad \mathbb{D}(\xi) =\frac{1}{\mu}\int_V v\otimes v q(\xi, v) dv.
\end{equation}
\begin{itemize}
\item[]{Equation \eqref{forwardlimit}
is a well studied anisotropic diffusion equation \cite{Risken1992}.
The mean first passage time $\tau(\xi)$ to the boundary of the spatial domain can be obtained by introducing the survival probability as shown in Section \ref{survprob} and following the same 
procedures here described. We do this below.} 
\end{itemize}
%
%
%

\subsection{Parabolic Scaling of the Backward Kolmogorov Equation}
\label{s:parabackward}
Under anisotropic, symmetric transport, the backward Kolmogorov equation for the survival probability $s(x,v,t)$
in \eqref{survive_dyn2} can be written as  
\[ s_t -v\cdot \nabla s = \mathcal{L}^* s .\]
Similarly to the forward case, we rescale space and time as in \eqref{scales} to obtain
\begin{equation}
\label{qbackwardscaled}
\ep^2 s_\tau -\ep v\cdot\nabla_\xi s =\mathcal{L}^* s. 
\end{equation}
We can now expand $s(\xi,v, \tau)$ in powers of $\ep$ so that
\[ s(\xi,v,\tau) =s_0(\xi,v,\tau) +\ep  s_1(\xi,v,\tau) +\ep^2 s_2(\xi,v,\tau) +\cdots\]
and compare orders of $\ep$ as follows

\vspace{0.2cm}
\noindent
\begin{itemize}
\item[\textbf{$\ep^0:$ } ] 
{Collecting the zeroth-order in $\ep$ terms in \eqref{qbackwardscaled} leads to  
\[
0= \mathcal{L}^* s_0. 
\]
Hence, $s_0 \in {\mbox {ker}} \, {\cal L}^*$ by Lemma \ref{l:kernels} and is independent of $v$:}
\end{itemize}
\begin{equation}\label{assumptions2}
s_0(\xi, v, \tau) = s_0 (\xi, \tau).
\end{equation}
\begin{itemize}
\item[$\ep^1:$ ] 
{Collecting the first-order in $\ep$ terms in \eqref{qbackwardscaled} leads to  
\[ -v\cdot\nabla_\xi s_0 = \mathcal{L}^* s_1.\]
We check the solvability condition on $\langle1\rangle^\perp$:
\[ (v\cdot\nabla_\xi s_0, 1) = \nabla_\xi\cdot \int_Vv dv \, s_0(\xi,\tau) =0.\]
The last identity arises from the assumption of symmetry in the velocity space $V$
so that if $v \in V$ then also $-v \in V$.  Lemma \ref{l:kernels2} 
implies}
\end{itemize}
\begin{equation}
\label{s1}
s_1(\xi, v, \tau ) = \frac{1}{\mu} v\cdot\nabla_\xi s_0(\xi,\tau). 
\end{equation}
\begin{itemize}
\item[$\ep^2:$ ] 
{Collecting the second-order in $\ep$ terms in \eqref{qbackwardscaled} leads to  
\[ s_{0\tau} -v\cdot\nabla_\xi s_1 = \mathcal{L}^* s_2 = -\mu s_2 +\mu \int_Vq(\xi, v') s_2(\xi,v',\tau) dv'.\]
We multiply this equation by $q(\xi, v)$, integrate over $V$, and express $s_1$  through \eqref{s1} to 
find
\begin{eqnarray*}
    s_{0\tau} -\int_V q(\xi,v) v\cdot\nabla_\xi \left(\frac{1}{\mu} v\cdot\nabla_\xi s_0(\xi,\tau) \right) dv 
 &=& \\ &&  \hspace{-4cm}    -\mu\int_V q(\xi,v) s_2(\xi, v,t) dv +\mu \int_V q(\xi,v) dv \int_Vq(\xi, v') s_2(\xi,v'\tau) dv'  = 0, 
\end{eqnarray*}
where we use the condition $\int_V q(x, v) dv = 1$ which also implies $\int_V q(\xi, v) dv = 1$. 
The above identity can be rewritten as 
\[ s_{0\tau} =\frac{1}{\mu}\int_V q(\xi,v)  v\otimes v dv :\nabla_\xi\otimes\nabla_\xi s_0.\]
Finally,  utilizing the same diffusion tensor as in  \eqref{forwardlimit} we  write 

\[ s_{0\tau} = \mathbb{D}(\xi):\nabla_{\xi}\otimes\nabla_{\xi} s_0, \]
which is the parabolic limit of the backward equation \eqref{survive_dyn2}. 
If we integrate this equation over time
from $0$ to $\infty$, and use the fact that $s_0(\xi,0)=1$ then we obtain the MFPT equation for $\tau(\xi)$}
\end{itemize}
\begin{equation}\label{eqn:MFPTJump}
-1 = \mathbb{D}(\xi) :\nabla_{\xi}\otimes\nabla_{\xi} \tau(\xi),
\end{equation}
%

Shifting back to the original spatio-temporal coordinates $(x,t)$, and using $T(x) = \varepsilon^2 \tau(\varepsilon x)$,
we find an expression for the mean first passage time $T(x)$ that is independent of 
the velocity $v$
\begin{equation}
\label{anisotropicMFPT3} 
-1 = \mathbb{D}(x) :\nabla\otimes \nabla\, T(x), \qquad \mathbb{D}(x) = \frac{1}{\mu} \int_V v\otimes v\, q(x,v) dv.
\end{equation}

\al{Note that if $\Theta$ is independent of $v$, 
this equation follows directly from \eqref{anisotropicMFPT}.}

\section{Applications}\label{s:applications}

To illustrate our results we consider several examples of anisotropic oriented movement in two-dimensional regions. We 
begin with the bimodal von-Mises distribution introduced in \eqref{bimodalvonMises} with $n=2$ to express the orientation biases of the random walker. Since in this case $V = \sigma {\Sn} = \sigma \Sone$, 
and $\omega = \sigma^{n-1} =\sigma$ from \eqref{norm}, we write
\begin{equation}\label{eq:vonMises}
q(x, v) = \frac{\tilde{q}(x, \theta)}{\sigma} = \frac{1}{4\pi \sigma I_0(k(x))}
\left( e^{k(x) \gamma(x)\cdot \theta} +e^{- k(x) \gamma(x)\cdot \theta)}\right) = \frac{\cosh({k(x) \gamma(x)\cdot \theta})}{2\pi \sigma I_0(k(x))},
\end{equation}
where $k(x) \geq 0$ is a measure of concentration and the 
unit vector $\gamma(x)$ defines the preferred orientation at location $x\in\Omega$. 
The bimodal von-Mises distribution includes both directions $\pm \gamma(x)$ 
to allow for symmetric movement along the preferred orientation. 
The functions $I_j(k(x))$ are modified Bessel functions of the first kind of order $j\geq 0$.  
Note that if $k(x)=0$ then $q(x,v)$ becomes the uniform distribution on a circle $q(x,v) = 1/(2\pi \sigma)$ and 
that if we let $k(x) \to \infty$ the distribution displays two $\delta$-singularities, one at $\theta=\gamma(x)$ and one at $\theta=-\gamma(x)$.

In this scenario, where particles move at fixed speed $\sigma$ and 
assuming the turning kernel biases motion along a preferred orientation $\gamma(x)$ 
that does not depend on incoming velocity, we assume that we are close to the parabolic limit, $\Theta(x,v) =T(x)$, and use \eqref{anisotropicMFPT3}. 
The diffusion tensor 
$\mathbb D(x)$  that appears in \eqref{anisotropicMFPT3} can be calculated using  \eqref{difftensor}
leading to \cite{vonMises} 
\begin{equation}
\label{eq:DiffVonMises}
 \mathbb{D} (x)= \frac{\sigma^2}{2\mu}\left(1-\frac{I_2(k(x))}{I_0(k(x))} \right) \mathbb{I} +\frac{\sigma^2}{\mu} \frac{I_2(k(x))}{I_0(k(x))} \gamma(x) \gamma^T(x).
\end{equation}
In the following examples, we show solutions of \eqref{anisotropicMFPT3} using
various choices for the domain $\Omega$ and the fields $\gamma(x)$ and $k(x)$ in \eqref{eq:DiffVonMises}.

\subsection{MFPT for Anisotropic Transport on a circular domain}
\label{sec:disk}

Bica et al \cite{Bica} studied anisotropic transport on a two-dimensional disk 
where the bimodal von-Mises distribution was used to describe biological particles moving
in radially symmetric environments. Applications include
intra-cellular transport along microtubules, cancer cell migration along collagen fibres, or
animals aggregating near watering holes. 
Here we similarly assume a two-dimensional circular domain $\Omega$ with radius $R_0$, 
and employ planar polar coordinates $ (r,\ff)$.  We now need to specify $k(x)$, $\gamma(x)$.
For most of this section we let the parameter of concentration be uniform,  $k(x) = k \geq 0$, 
however we keep general $k(x)$ in our analytical derivations. 
We also assume that the preferred orientation $\gamma(x)$ is either radial, {\textit {i.e.}} 
as if along the spokes of a bicycle wheel, or circular, along the direction perpendicular
to the radial direction, following concentric circles inside the disk. 
If we now write the radial unit vector as $\hat x = x/||x|| = (\cos \ff , \sin \ff) $
and its orthogonal as $\hat x^{\perp} = ( - \sin \ff , \cos \ff) $,  then
  $\gamma(x)$ can be represented as
\begin{equation}
\gamma(x) = \left\{\begin{array}{ll} 
\hat x 
& \mbox{for radial orientation,} \\[5pt]
\hat x^{\perp} 
& \mbox{for circular orientation.} 
\end{array}\right.
\end{equation}
These two choices can be combined into a common notation by introducing 
the anisotropy indicator $\alpha(x)$
 as follows \cite{Bica} 
\bsub
\label{eqn:DiffTensor}
\begin{equation}\label{radialD}
\mathbb{D}(x) = \frac{\sigma^2}{\mu} \left\{\begin{array}{lcl}
\frac{1}{2}(1-\alpha(x))\mathbb{I} +\alpha(x) \hat x \hat x^T &\qquad& x\neq 0\\[1ex]
\frac{1}{2} \mathbb{I} && x =0. \end{array}\right.
\end{equation}
where $\alpha(x)$ for general $k(x)$ is given by
\begin{equation}
\label{radialD2}
\alpha(x) = \frac{I_2(k(x))}{I_0(k(x))} \times \left\{\begin{array}{rl} 
+1 & \mbox{for radial orientation,} \\[5pt]
-1 & \mbox{for circular orientation.}
\end{array}\right.
\end{equation}
\esub
If we set $k(x) = k$ then $\alpha(x) = \alpha$ is also uniform, and since $I_2(k) \leq I_0(k)$
for all $k \geq 0$ it is also true that $\alpha \in (-1,1)$.  Note that the uniform distribution
that arises from setting $k(x) = 0$ in \eqref{eq:vonMises}
corresponds to $\alpha(x) = 0$ due to  \eqref{radialD2} and
the properties of the Bessel functions.  The limit $k(x) \to \infty$, 
which leads to a singularity along $\pm \gamma(x)$ 
as discussed above, results in $\alpha (x) \to 1$ 
for radial orientation and $\alpha (x) \to -1$ for circular orientation. This implies
that motion is singularly biased along the radial direction for
$\alpha \to 1^{-}$ and along the tangential direction for  $\alpha \to -1^{+}$. 
Finally, in \eqref{radialD}
$\mathbb{I}$ is the identity matrix and the tensor product $\hat x \hat x^T$ 
 is given by
\[ \hat x \hat x^T = \left(\begin{array}{cc} \cos^2\ff & \sin \ff \cos\ff \\ \sin\ff\cos\ff & \sin^2\ff \end{array}\right). \]
We can now solve the anisotropic mean first passage time equation \eqref{anisotropicMFPT3}. 
Given the radial symmetry, we use the notation 
$T_{\alpha}(r)$ to indicate the MFPT for the random walker starting at $r$ under the 
anisotropy indicator $\alpha$.
A direct calculation of \eqref{anisotropicMFPT3} and \eqref{eqn:DiffTensor} 
assuming $\alpha(x) = \alpha(r)$ is only radially dependent,
yields
\begin{equation}\label{eqn:MFPT_radial}
\frac{1+\alpha(r) }{2r} \, \Big( r T'_{\alpha } (r)\Big)'  -
 \frac{\alpha (r)}{r}  T'_{\alpha}(r) = -\frac {1}{ 2D}
\end{equation}
where the diffusion constant $D = \sigma^2/ (2 \mu)$  is consistent with 
the $n=2$ diffusivity introduced in \eqref{ndimdiff}. 
Note that the case $\alpha =0$ corresponds to the standard diffusion process in 
two-dimensions. 
The general solution of \eqref{eqn:MFPT_radial} is
\begin{equation}
\label{disk1}
T_{\alpha}(r) = -\frac{r^2}{4D} + H_1 + H_2 \int^r \exp\left[-\int^{\eta} \frac{1}{s} \frac{1-\alpha(s)}{1+\alpha(s)} d s\right] d\eta,
\end{equation}
where $H_1, H_2$ are constants that depend on the chosen boundary conditions. 
For uniform $\alpha(s) = \alpha$, \eqref{eqn:MFPT_radial}
reduces to 
\begin{equation}
\label{disk2}
T_{\alpha}(r) =  -\frac{r^2}{4D} + H_1 + H_2 r^{\frac{2\alpha}{1+\alpha}}.
\end{equation}
Below, we solve for $H_1, H_2$ 
in given geometric scenarios 
and for fixed $\alpha \in (-1,1) $. Note that
since \eqref{disk2} contains two unknowns, two conditions are necessary to identify them.

\subsubsection{Exit from a disk:} 
In this scenario we assume the random walker can exit the disk only through the boundary
at $r = R_0$. This implies that the mean first exit time at $r=R_0$ must be zero. We also impose smoothness at the origin so that all derivatives at $r=0$ be zero. 
Thus, we solve \eqref{disk2} in the domain $0<r<R_0$ with boundary conditions
\begin{equation}
T'_{\alpha}(0) = 0, \quad T_{\alpha}(R_0) =0,
\end{equation}
to obtain the solution
\begin{equation}\label{solutionlong}
T_{\alpha}(r) =  \frac {1}{4D} \big( R_0^2 - r^2 \big).
\end{equation}
Remarkably,  \eqref{solutionlong} is independent of $\alpha$. 
Thus, the two-dimensional MFPT to reach the boundary of a disk 
for a particle moving under the von-Mises jump distribution \eqref{eq:vonMises} 
with a uniform concentration parameter
$k(x) = k$ has the same form as in the case of a random, diffusive process.

\begin{figure}[t]
\centering
    \subfigure[Exit from inner circle.]{\includegraphics[width = 6cm]{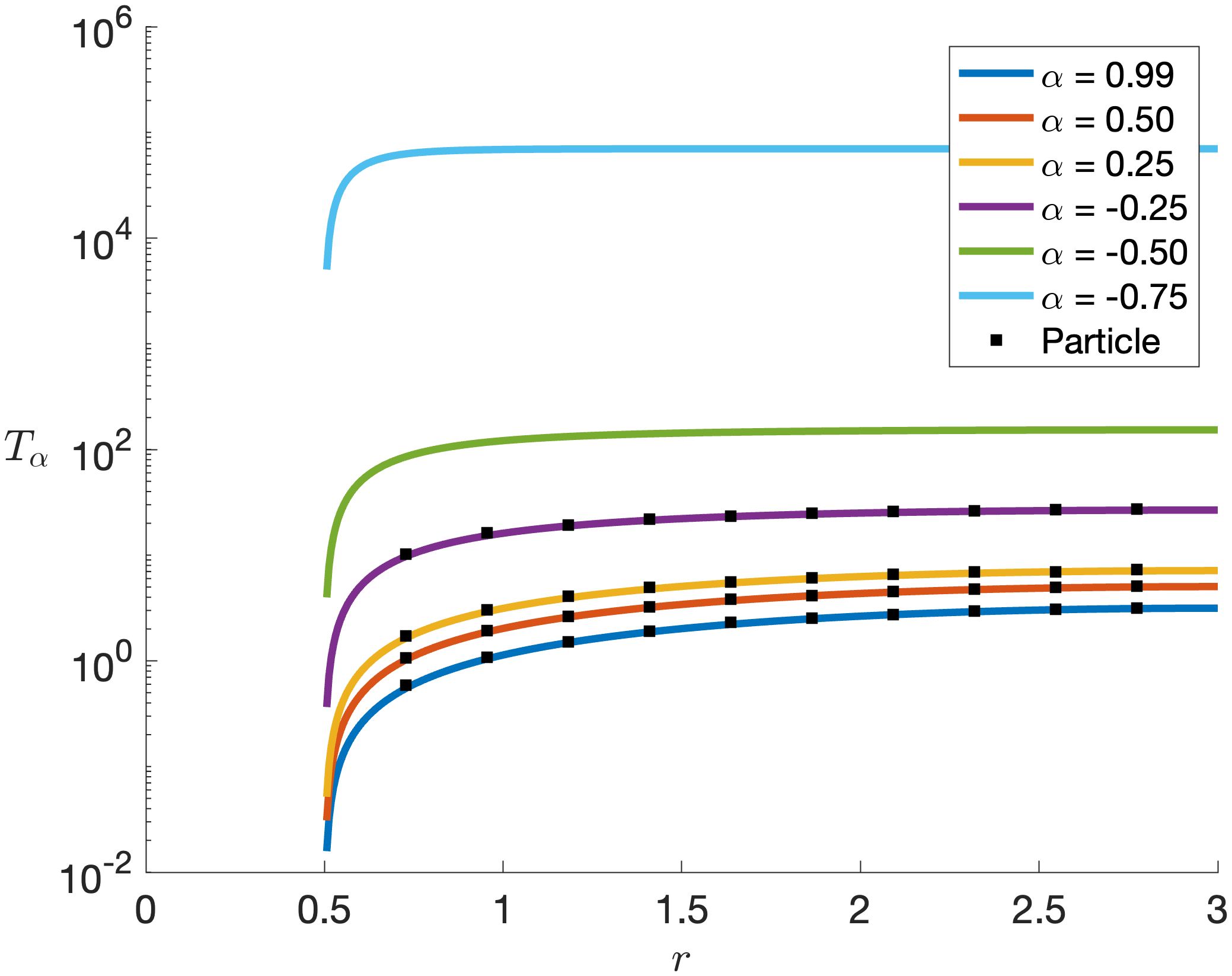}}\qquad
    \subfigure[Exit from inner or outer circle.]{\includegraphics[width = 6cm]{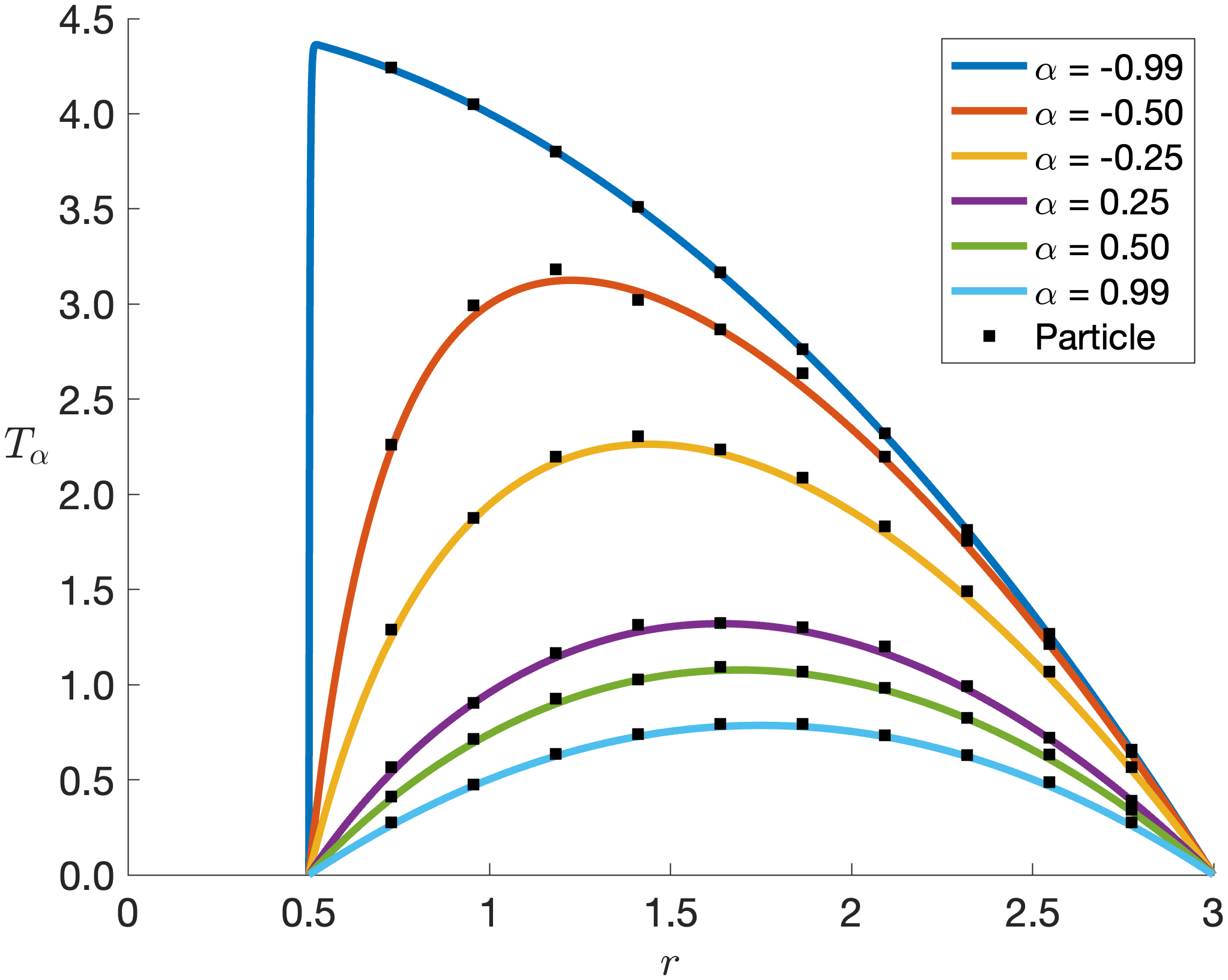}}
    \caption{The MFPT \eqref{eqn:MFPT_radial} for $D=0.5$ on an annulus with domain $\rho<r<R_0$, inner circle radius $\rho = 0.5$, outer circle radius $R_0 = 3$ and several values of $\alpha \in (-1,1)$.  Panel (a): Particles exit the annulus only through the inner circle; the corresponding boundary conditions under which to solve \eqref{disk2} are $T_\alpha(\rho) = T'_\alpha(R_0)=0$. Note that $T_\alpha(r)$ increases with $\alpha$ for given $r > \rho$ and $T_\alpha(r > \rho)\to\infty$ as $\alpha \to -1^{+}$ as shown in \eqref{circinner}. Under purely circular motion particles move tangentially and cannot decrease their distance from the origin, hence they will never exit the annulus.    
    Panel (b): Particles may exit the annulus 
    through both the inner and outer circles; the corresponding boundary conditions under which to solve \eqref{disk2} are $T_\alpha(\rho) = T_\alpha(R_0)=0$. In the limit $\alpha\to-1^+$ the solution $T_\alpha(r)$ displays a discontinuity at $r=\rho$ as particles located at $r \to \rho^{+}$ will be able to exit the annulus only through the outer circle at $r= R_0$. Particle simulations (black squares) based on $N=10^4$ independent stochastic simulations of \eqref{kinetic1} with $\mu = 10^4$, $\sigma = \sqrt{\mu}$ agree with theoretical predictions. Specialized numerical algorithms are needed to simulate the exit from the inner circle as $\alpha \to {-1}^{+}$ since $T_{\alpha}\gg1$. Indeed, we note the difference in the scale of the vertical axes between the two panels. All quantities are in arbitrary units.}
    \label{fig:MFPT_annulus}
\end{figure}

\subsubsection{Exit from the inner circle of an annulus:} 
We now consider motion constrained to an annular region so that 
 $\rho<r<R_0$ for all trajectories. As shown below,  
the effects of anisotropy can be more pronounced here compared to motion in a disk.
We specifically assume that the random walker can only exit the annulus via the inner circle
at $r = \rho$ and impose reflecting (Neumann) boundary conditions at the outer circle at 
$r = R_0$.  Thus we solve \eqref{disk2} in the domain $\rho<r<R_0$ with boundary conditions
\begin{equation}
T_{\alpha}(\rho) = 0, \quad T'_{\alpha}(R_0) = 0.
\end{equation}
The general solution of \eqref{eqn:MFPT_radial} under uniform $\alpha$ is
\begin{equation}
\label{MFPTgeneral}
T_{\alpha}(r) = \frac 1 {4 D}  (\rho^2 - r^2) + 
\frac {1}{2 \beta D  } R_0^{\frac {\beta} {\alpha }} 
(r^{\beta} - 
\rho ^{\beta}), \qquad {\mbox {with}} \qquad \beta \equiv \frac{2\alpha}{1+\alpha}.
\end{equation}
We then derive the limits for the MFPT 
corresponding to motion that is 
purely radial ($\alpha \to 1^{-}$), isotropic ($\alpha \to  0$), 
 or purely circular ($\alpha \to -1^{+}$) 
\bsub
\begin{align}
\label{radinner}
\lim_{\alpha\to {1^{-}}} T_\alpha (r) &=  \frac{1}{4D}(\rho^2 - r^2) + \frac{R_0}{2D}(r -\rho),\\[4pt]
\label{neuinner}
\lim_{\alpha\to 0} T_\alpha (r) & =\frac{1}{4D}(\rho^2 - r^2) + \frac{R_0^2}{2 D} \log\Big(\frac{r}{\rho}\Big), \\[4pt]
\label{circinner}
\lim_{\alpha\to -1^{+}} T_\alpha (r) & = \left\{ \begin{array}{cl} +\infty & r>\rho;\\[4pt]
0 & r = \rho.\end{array} \right. 
\end{align}
\esub
Since the function $r - \rho - R_0 \log (r/\rho)$ is a negative, decreasing function of $r$
for $\rho <  r < R_0$ the MFPT in the purely radial case 
$(\alpha \to 1^{-}$ in  \eqref{radinner})
is less than in the isotropic case $(\alpha \to 0$ in \eqref{neuinner}). 
This implies that exiting the annulus from its inner circle
is faster if the motion is constrained along the radial direction rather than being isotropic. Also note that under purely circular motion ($\alpha \to -1^{+}$ in \eqref{circinner}) the MFPT $T_{\alpha} (r > \rho) = +\infty$
for any starting position $r$.  Since the preferred motion is along the tangent direction, a particle that started at $r > \rho$ can never randomly turn to smaller values of $r$ 
to reach the only exit boundary at $r=\rho$. Hence, it will remain trapped in the annulus for infinite time.

\subsubsection{Exit from the outer circle of an annulus:} 
We now consider the opposite scenario where the random walker can only exit the annulus via the outer 
circle at $r = R_0$ and impose reflecting (Neumann) boundary conditions at the inner circle at 
$r = \rho$.  Thus we solve \eqref{disk2} in the domain $\rho<r<R_0$ with boundary conditions
\begin{equation}
T'_{\alpha}(\rho) = 0, \quad T_{\alpha}(R_0) = 0.
\end{equation}
The general solution of \eqref{eqn:MFPT_radial} under uniform $\alpha$ is given by replacing
$\rho$ and $R_0$ in \eqref{MFPTgeneral}

\begin{equation}
T_{\alpha}(r) = \frac 1 {4 D}  (R_0^2 - r^2) + 
\frac {1}{2 \beta D  } \rho^{\frac {\beta} {\alpha }} 
(r^{\beta} - 
R_0 ^{\beta}), \qquad {\mbox {with}} \qquad \beta \equiv \frac{2\alpha}{1+\alpha}.
\end{equation}
The limits for the MFPT 
corresponding to motion that is 
purely radial ($\alpha \to 1^{-}$), isotropic ($\alpha \to  0$), 
 or purely circular ($\alpha \to -1^{+}$) are
\bsub
\begin{align}
\label{radouter}
\lim_{\alpha\to {1^{-}}} T_\alpha (r) &=  \frac{1}{4D}(R_0^2 - r^2) + \frac{\rho}{2D}(r -R_0),\\[4pt]
\label{neuouter}
\lim_{\alpha\to 0} T_\alpha (r) & =\frac{1}{4D}(R_0^2 - r^2) + \frac{\rho^2}{2 D} \log\Big(\frac{r}{R_0}\Big), \\[4pt]
\label{circouter}
\lim_{\alpha\to -1^{+}} T_\alpha (r) & = \frac{1} {4 D} (R_0^2 - r^2).
\end{align}
\esub
In this case, the function $r - R_0 - \rho  \log (r/R_0)$ is a negative, increasing function of $r$
for $\rho <  r < R_0$ so that the MFPT in the purely radial case $(\alpha \to 1^{-}$ in
\eqref{radouter})
is still less than in the isotropic case $(\alpha \to 0$ in \eqref{neuouter}). 
Thus, exiting the annulus from its outer circle
is faster if particle motion is constrained along the radial direction rather than being isotropic. Since $r < R_0$, comparing 
\eqref{circouter} to \eqref{radouter} and \eqref{neuouter}
shows that purely circular motion yields the longest MFPT
to exiting the annulus. Furthermore for $\alpha \to -1^{+}$ 
the MFPT does not depend on $\rho$ since
the motion is purely tangential and, as discussed above, a particle starting
at $r > \rho$ will always randomly increase its distance from origin. 

Upon comparing \eqref{radinner} and \eqref{radouter}, one can show that under purely radial motion, the MFPT to exit the annulus from the outer or inner circle is the same if trajectories start at $r = r_{\rm c}= (\rho + R_0)/2$, the exact midpoint. For 
$r > r_{\rm c}$ the MFPT is shortest if exiting from the outer circle, for 
$r < r_{\rm c}$ the MFPT is shortest if exiting from the inner circle.

\subsubsection{Exit from the inner or outer circle of an annulus:} 
Finally, we solve \eqref{disk2} in the case where a particle can exit the annulus from both inner and outer circles using the following boundary conditions
\[
T_{\alpha} (\rho) = T_{\alpha}(R_0) = 0.
\]
We find
\begin{equation}
    T_{\alpha}(r) = - \frac 1 {4 D} r^2 + \frac 1 {4 D} 
\left( \frac{R_0^{\beta} \rho^2 - R_0^2 \rho^{\beta}}{R_0^{\beta} - \rho^{\beta}} \right)
+  
 \frac {r^{\beta}}{4D} \left( \frac{R_0^2 - \rho^2}{R_0^{\beta} - \rho^{\beta}}
\right),  \qquad \beta = \frac{2\alpha}{1+\alpha}.
\end{equation}
As done above, we highlight the limiting cases 
of purely radial ($\alpha \to 1^{-}$),
isotropic ($\alpha = 0$),
and purely circular ($\alpha \to -1^{+}$) motion to find
\bsub
\begin{align}
\label{radboth}
\lim_{\alpha\to {1^{-}}} T_\alpha (r) &= \frac{1}{4D}(\rho - r)(r-R_0),\\[4pt]
\lim_{\alpha\to 0}  T_\alpha (r) &= \frac 1 {4 D}
\frac {(R_0^2 - \rho^2) \log r - (r^2 - \rho^2) \log R_0 
+ (r^2 -R_0^2) \log \rho} { \log R_0 - \log \rho },\\[4pt]
\label{circboth}
\lim_{\alpha\to -1^{+}} T_\alpha (r) & = \left\{ \begin{array}{cl} 
\displaystyle{\frac 1 {4D}}
(R_0^2 -r^2)& r>\rho;\\[4pt]
0 & r = \rho.\end{array} \right. 
\end{align}
\esub
By comparing \eqref{radboth} 
with \eqref{radinner} and \eqref{radouter},  one can show 
that under purely radial motion the MFPT to exit the annulus from either the inner or 
outer circle is always less compared to when exit is allowed from only one of the two boundaries. Note the discontinuity in \eqref{circboth}: under purely circular motion particles starting
at $r > \rho$ can exit the annulus through the outer circle so we recover
the same result as in \eqref{circouter}.

\subsection{MFPT for anisotropic transport with linear features.}
\label{sec:linear}

\begin{figure}[t]
\centering
    \includegraphics[width=1.0\linewidth]{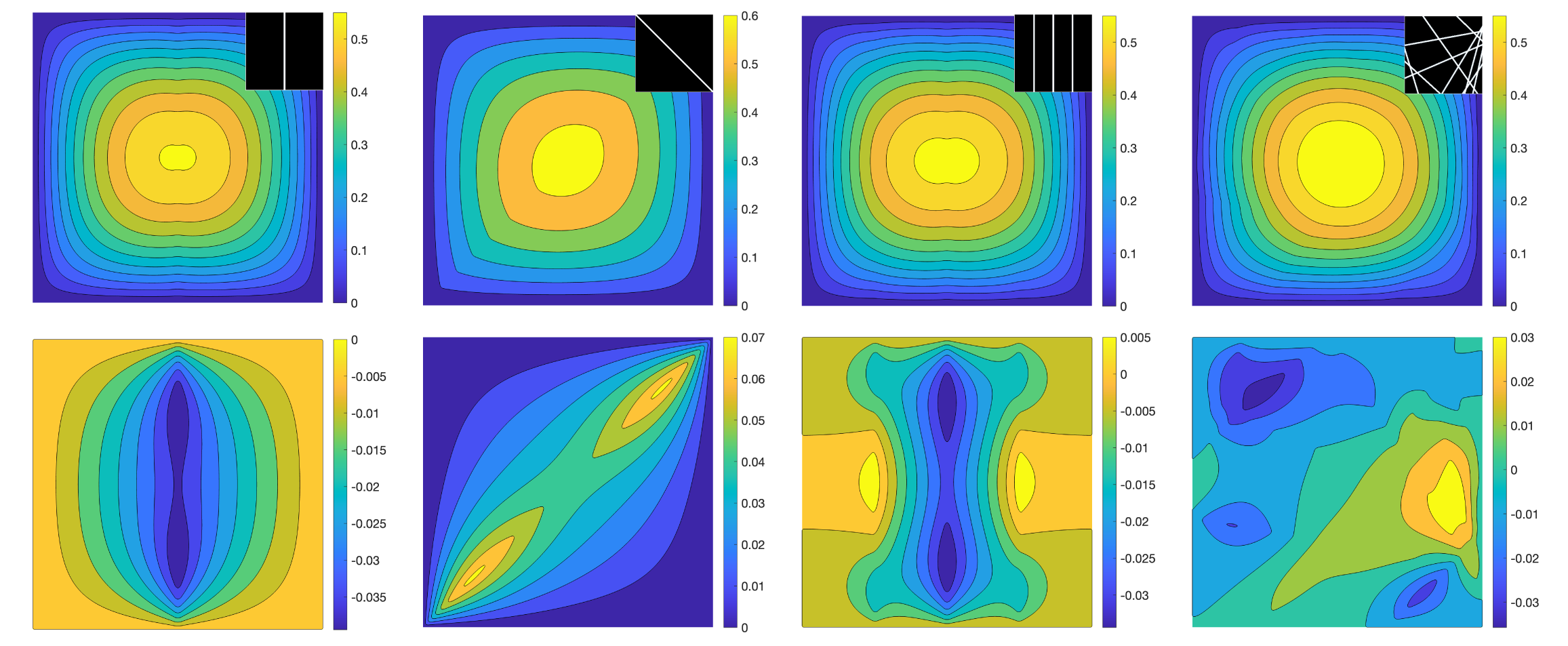}
    \caption{Top row: The MFPT to the boundary of the domain $\Omega = [-1,1]^2$
    in the presence of one or multiple linear segments orienting particle motion.
    Results are obtained by numerically solving \eqref{anisotropicMFPT3} for $T(x)$ in the domain $\Omega = [-1,1]^2$ under the bimodal von-Mises distribution \eqref{eq:vonMises} modulated by the linear structures 
    shown in the inset as described in \eqref{eq:linear_eqn}. The cutoff distance $d_0$ for particles to move along the linear structures is set at $d_0 = 0.02$ and the corresponding measure of concentration is 
    $k_0=25$.
    The anisotropy profiles from left to right are a vertical line, a slant line, 3 parallel vertical lines and 10 randomly intersecting lines. Bottom row: 
    The difference between the MFPT under anisotropic transport and the MFPT for a random walker in the same two-dimensional domain. The scenario of 10 randomly intersecting segments in the right-most panels yields the smallest difference between the two MFPTs. All quantities are in arbitrary units.    
    \label{fig:MFPT_linear}}
\end{figure}

As discussed previously, anisotropic transport 
is often facilitated by the presence of linear structures in the environment. 
For example, roads and seismic lines modulate animal movement in ecological 
landscapes while intra-cellular transport occurs along actin filaments or
microtubules. Multiple segments (of roads, or filaments) carry random relative orientations and occasionally intersect.  As such, we study the MFPT in the presence
of one or several preferential orientations. For concreteness, we use a square domain $\Omega$ and assume that particles can exit the domain
through any of its edges. We thus solve \eqref{anisotropicMFPT3} with
absorbing (Dirichlet) boundary conditions.  To model the presence of linear features we 
introduce a bimodal von-Mises distribution constructed using \eqref{eq:vonMises} and a tessellation process. For each point $x$ we find the Euclidean distance $d(x)$ to the closest linear feature and evaluate the anisotropic distribution in \eqref{eq:vonMises} by setting $k(x), \alpha(x)$ as follows
\begin{equation}\label{eq:linear_eqn}
k(x) = \left\{ \begin{array}{rl}
k_0, & d(x) < d_0\\[4pt] \qquad
0, & \mbox{otherwise} 
\end{array} \right., 
\qquad  \alpha(x) = \frac{I_2(k(x))}{I_0(k(x))}.
\end{equation}
The directional field $\gamma(x) = (\cos\ff,\sin \ff)$ is taken as the slope of the nearest line to point $x$. 
These choices imply that once a particle is at a distance greater
than a cutoff value $d_0$ from any linear feature, {\textit {i.e.}}
for $d(x) \geq d_0$, movement is an isotropic random walk since in this case $k(x) = \alpha(x) = 0$. Once the particle is closer to a linear feature, for $d(x) < d_0$, movement is biased towards a bidirectional random walk along the
$\pm \gamma(x)$ directions corresponding to the specific linear feature as per \eqref{eq:vonMises}.  The larger the value of $k_0$, the larger the 
bias along $\pm \gamma(x)$.  
Note that 
setting $\alpha(x)$ positive in \eqref{eq:linear_eqn} corresponds to radial orientation in \eqref{radialD2} and that $k_0 \to \infty$ yields $\alpha (x: d(x) < d_0) \to 1$.  
The contour plots shown in the top row of Fig.~\ref{fig:MFPT_linear} are the numerical solutions of \eqref{anisotropicMFPT3} 
and \eqref{eq:linear_eqn} using linear segments given by a vertical line, a slant line, 3 parallel vertical lines and 10 randomly intersecting lines as shown in the inset and using $k_0=25$ and $d_0 = 0.02$ which correspond to $\alpha( x: d(x) < d_0) = 0.922$. 
For these parameters the effective motion is a combination of isotropic two-dimensional ($d(x) \geq d_0$) and quasi one-dimensional random walks $(d(x) < d_0)$. 
The modulation of the MFPT due to the simple linear features is clearly distinguishable in all four scenarios shown in Fig.~\ref{fig:MFPT_linear}. We also numerically evaluate the MFPT in the same square geometry but under isotropic diffusion, which corresponds to setting $d_0 \to 0$ in \eqref{eq:linear_eqn}, and show the difference between 
the two MFPTs in the bottom row of Fig.~\ref{fig:MFPT_linear}.
The right-most panels in Fig.~\ref{fig:MFPT_linear} with 10 randomly oriented linear features yield an MFPT that is closest to an isotropic random walker.  

\subsection{MFPT on an ecological landscape}\label{sec:wolf}

Finally, we consider an actual ecological landscape with linear features to specifically 
model animal movement using the results and methods of Section \ref{sec:linear}. McKenzie 
et al.\,\cite{MCKENZIE2009,McKenzie2012} have previously used the anisotropic model \eqref{anisotropicMFPT3} and fitted it to oriented movement of red foxes in Minnesota in the presence of prey and a den site \cite{MCKENZIE2009}, and to wolf movement data in forest areas of the Canadian Rocky Mountains, where forests are disturbed by seismic lines \cite{McKenzie2012}. They find that preferred movement of wolves along seismic lines increases the encounter rate of predator and prey, thereby reducing prey density. 

Here we revisit this case as an illustration and use an aerial image
that was not used in \cite{McKenzie2012}. Fig.\,\ref{fig:Wolf_a} shows
a landscape of intersecting roads and seismic lines in the 
boreal forest of Western Canada. In \cite{vonMises,hillen2013transport} the forward transport equation \eqref{kinetic2} was solved on the rectangular domain of this image. Here we use the same image to evaluate the MFPT to the boundary of the domain.
The image was first digitized and thresholded such that 
landscape features are represented by a boolean variable. 
The  binarized image is then processed into two 
fields, $\gamma(x)$ and $d(x)$, representing the direction of, and distance to, the 
closest linear feature for each $x \in \Omega$, where $\Omega$ is the rectangular domain 
of the image. The resulting linear features and their associated directions are shown in
Fig.~\ref{fig:Wolf_c}. The directional field is incorporated into the bimodal Von-Mises distribution \eqref{eq:vonMises} so that motion near linear features is biased along the two directions $\pm\gamma(x)$. In the blank space of Fig.~\ref{fig:Wolf_c}, where directional information is not shown, the diffusion is isotropic.  

\begin{figure}[t]
    \centering
    \subfigure[Forest image.]{\includegraphics[width=4.5cm]{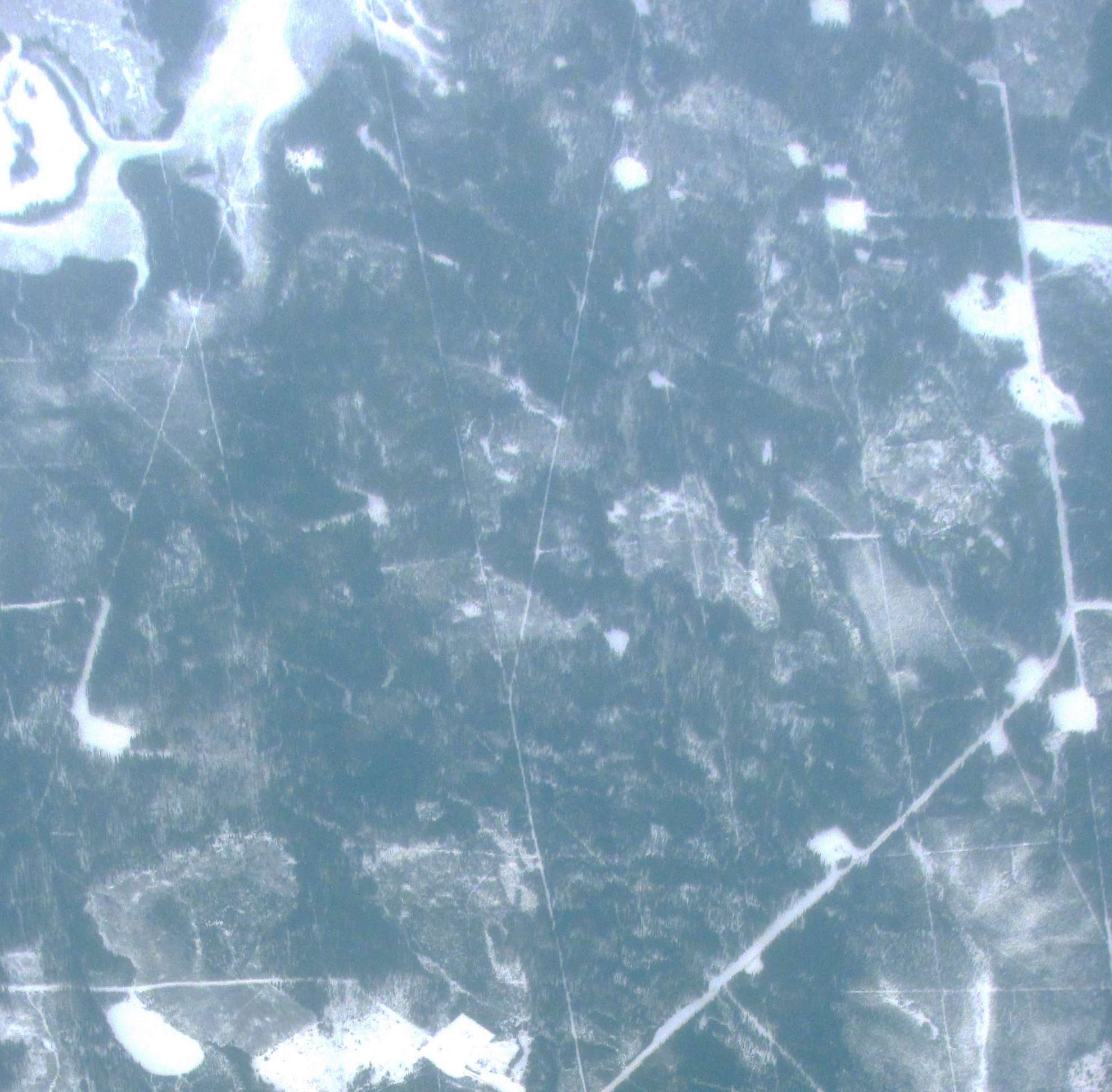}\label{fig:Wolf_a}}\qquad
    \subfigure[Linear features with orientation.]{\includegraphics[width=4.3cm,height=4.3cm]{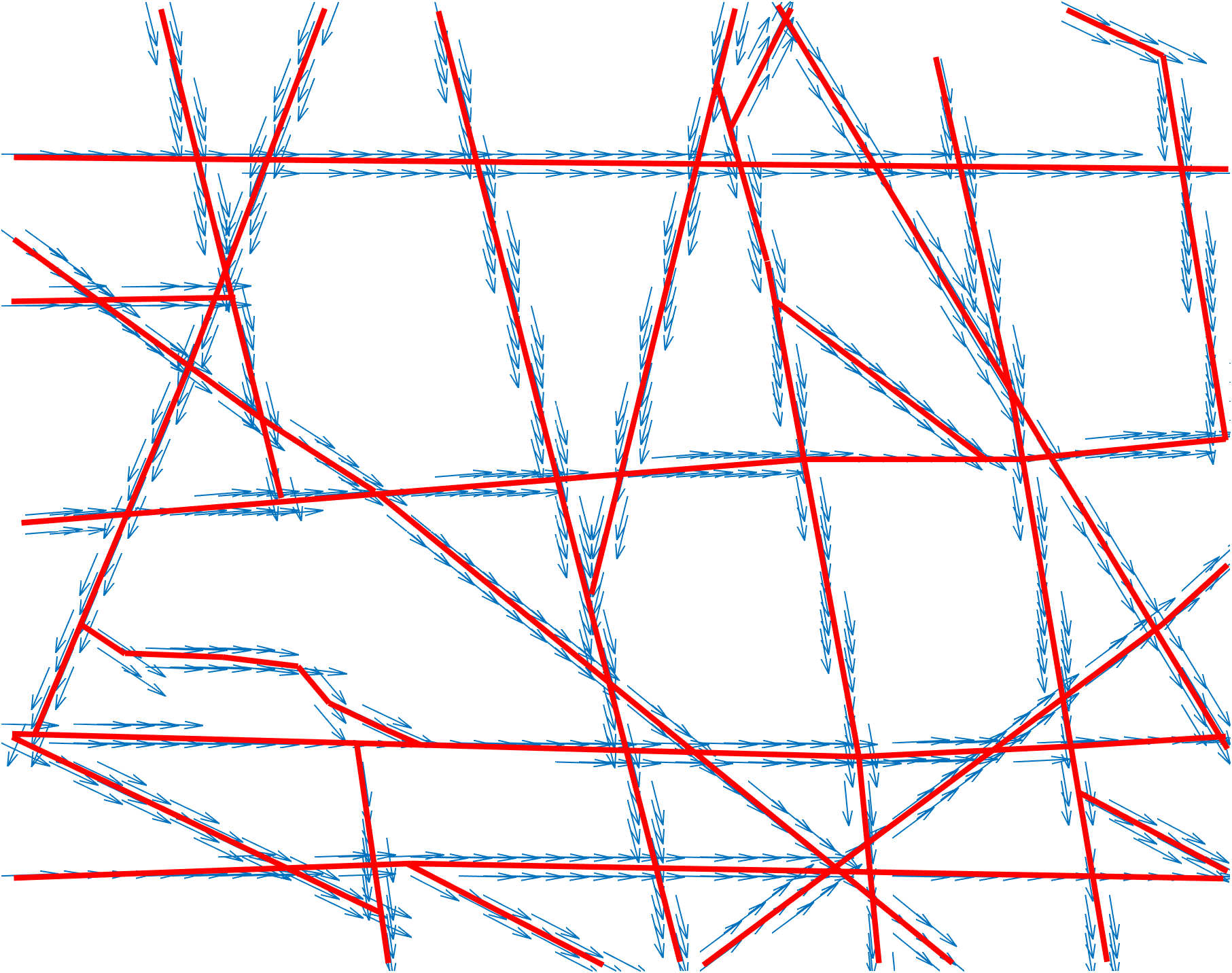}\label{fig:Wolf_c}}\quad\, \, 
    \subfigure[Normalized MFPT $T(x)/ T_{\rm max}.$]
    {\includegraphics[width=4.5cm,height=4.5cm]{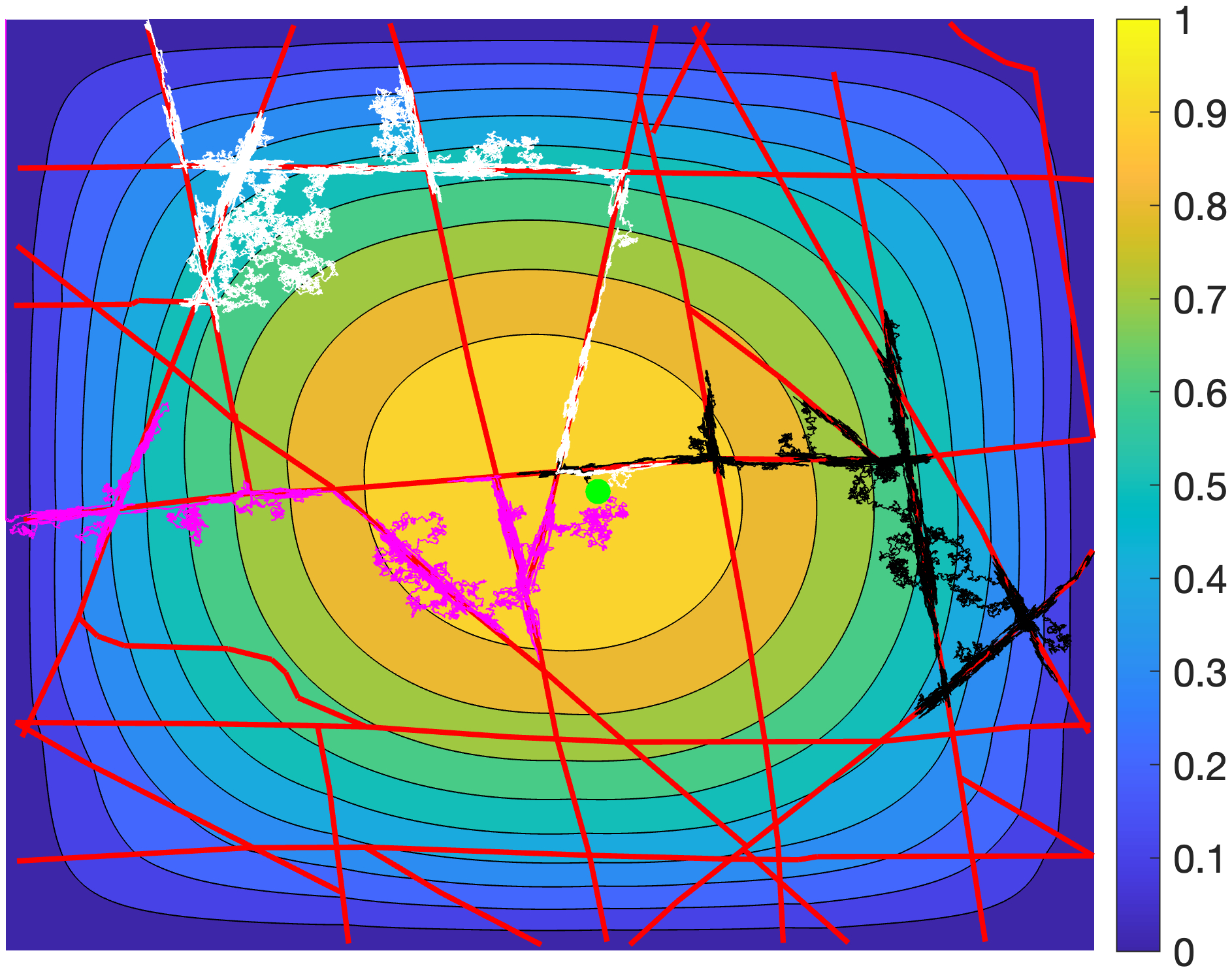}
    \label{fig:Wolf_d}}
    \caption{Calculation of the MFPT on an ecological landscape. Panel (a): Aerial snapshot of the boreal forest in Western Canada. Panel (b): A digitized and thresholded version of the same image highlighting roads and seismic lines with  bidirectional information. Panel (c): Contour plot of the MFPT as derived from
    \eqref{anisotropicMFPT3}, \eqref{eq:DiffVonMises}  and \eqref{eq:linear_eqn} using the method described in Section \,\ref{sec:linear}. Also shown are three single trajectories from particle simulations of \eqref{kinetic1}, each initialized at the origin with parameters $\mu = 10^4$, $\sigma = 10^4$ and bimodal von-Mises turning kernel \eqref{eq:vonMises}. Particle trajectories remain closely anchored to the linear segments before eventual absorption at different edges of $\partial\Omega$. All quantities are in arbitrary units.}
    \label{fig:Wolf}
\end{figure}

Similarly to what done in Section \,\ref{sec:linear}, the MFPT equation \eqref{anisotropicMFPT3} is solved on the image's rectangular domain with $k(x)$ given in \eqref{eq:linear_eqn} and the direction field shown in Fig.\,\ref{fig:Wolf_c}. 
We allow the random walker to exit the domain through any of the edges of the image and apply absorbing (Dirichlet) boundary conditions on $\partial\Omega$. 
The contour plot of the resulting numerical solution of \eqref{anisotropicMFPT3} is shown in Fig.\,\ref{fig:Wolf_d}. We observe a similar phenomenon to that observed in Fig.\,\ref{fig:MFPT_linear} where the superposition of linear features with different angles yields an apparently isotropic MFPT.  Fig.\,\ref{fig:Wolf_d} also displays three single trajectories obtained from stochastic simulations of \eqref{kinetic1} which shows strong alignment on the linear features. Hence, the occupation density of particles is strongly localized on the anisotropic features, while the MFPT is modulated to by these features to a much lesser degree. This, of course, is because the MFPT is the average
of the mean first passage time derived from all possible trajectories. 

\section{Conclusions}\label{s:conclusion}
Kinetic transport equations form a large class of models that are used in many applications from physics, engineering, material sciences. Here we focus on biological applications, which include a wide variety species and spatio-temporal scales such as {\textit{E.\, coli}} run and tumble movement \cite{OthmerXue}, brain tumor invasion \cite{SwanHillen}, sea turtle orientation \cite{PaHiTurtle}, 
DNA repair mechanisms \cite{Fok2008}, and many more. 

As it turns out, the general transport equation \eqref{kinetic1} is too general to derive a meaningful mean first passage time equation. Hence, we focused 
on three cases: the {\it diffusive}, {\it anisotropic}, and {\it parabolic limit} cases, all of which are relevant to biology. Each of these cases translates into a set of assumptions on the model parameters and in particular on the integral kernel $K(x,v,v')$. We summarize these assumptions in Table \ref{tab:assumptions} for accessible comparison. 
The diffusive and anisotropic cases are not disjoint, since, for example, the constant kernel $K(x,v,v')=|V|^{-1} $ is contained in both. 

\begin{table}[]
\footnotesize
    \centering
    \renewcommand{\arraystretch}{1.5}
    \begin{tabular}{l|c|l}
    \hline
     {\bf Case} & {\bf Reference}    & {\bf Assumptions}  \\[1ex]
    
\hline
General assumption & (\ref{assumptionsT}) &  $K\geq 0,\quad K \in L^2(V\times V),\quad \int_VK(x, v,v') dv =1.$ \\[1ex]
\hline
Diffusive case & (\ref{assumptionsT}) + (\ref{diffusiveK}) & $ K(x,v,v')=K(x,v,-v'),\quad  \int_V K(x,v,v') dv' =1$.\\[1ex]
\hline
Anisotropic case & (\ref{assumptionsT}) + \eqref{T3assumpt} & $K(x,v,v') = q(x,v),\quad   q(x,-v) = q(x,v)$.\\[1ex]
\hline
Parabolic scaling & (\ref{assumptionsT}) + \eqref{T3assumpt} & $\tau =\ep^2 t, \quad \xi=\ep x.$\\
&+ (\ref{scales},\ref{expansion}) &  $  p(\xi,v,\tau) =p_0(\xi,v,\tau) +\ep  p_1(\xi,v,\tau) +\ep^2 p_2(\xi,v,\tau) +\cdots$\\[4pt]
    \hline
    \end{tabular}
   
    \caption{Summary of the basic assumptions for the various cases studied in this manuscript.} 
    \label{tab:assumptions}
\end{table}

In our analysis we find two types of MFPT equations. For the diffusive and anisotropic cases we find an integro-PDE for the MFPT $\Theta(x,v)$, \eqref{generalMFPT} and \eqref{anisotropicMFPT}, respectively. The expected exit time $\Theta(x,v)$ depends on the initial location $x$ and the initial velocity $v$, and the explicit dependence on $v$ makes this equation so interesting. A systematic study of equations of the type \eqref{generalMFPT} and \eqref{anisotropicMFPT} has not been done yet. 
In the parabolic scaling case, the $v$-dependence disappears, and we obtain the anisotropic MFPT equation \eqref{anisotropicMFPT3} 
for $T(x)$. This model fits within the classical theory of MFPT equations, as, for example, illustrated in \cite{WEISS1967, Risken1992, VanKampen2007, REDNER2001}, and much is known about its properties and its solutions. 

Solutions to MFPT equations depend, of course, on the domain $\Omega$ and the boundary conditions on $\partial \Omega$. For the parabolic limit case \eqref{anisotropicMFPT3}, boundary conditions can be included in a standard way, using $T(x)=0$ on the absorbing part of the boundary and $\hat n \cdot \mathbb{D}(x) \nabla T(x) =0$ on the no-flux part of the boundary, where $\hat n$ denotes an outward normal vector on the boundary. The domain can be connected or disconnected, it can have holes and other special shapes, and the literature on various domains is vast (for example \cite{REDNER2001,GREBENKOV2016,CHOU2014,ward2010a,Venu2015}). 

For the diffusive or the anisotropic case, however, boundary conditions need to include the velocity $v$. A reasonable choice to describe absorbing boundaries is to assume that a random walker on the boundary $\partial \Omega$ that starts in an outward direction, is lost immediately. In mathematical terms, for any $x\in \partial\Omega$ such that $\hat n(x)$ denotes the unit outward normal vector, we have
\begin{equation}\label{Dbc}
\Theta(x,v) = 0\quad \mbox{for all}\quad v\in V: \; v\cdot \hat n>0. 
\end{equation}
On the reflecting parts of the boundary various conditions can be stipulated. They often take the form of an integral relation \cite{schwetlick}. Given $x\in \partial \Omega$, for each $v\in V$ with $v\cdot \hat n(x)\leq 0$, we define 
\begin{equation}\label{Nbc}
\Theta(x,v) = \int_{\{v'\cdot\hat n>0\}} I(x,v,v') \Theta(x,v') dv',
\end{equation}
with an appropriate integral kernel $I$. 
Combining the MFPT equations \eqref{generalMFPT} and \eqref{anisotropicMFPT} with these type of boundary conditions \eqref{Dbc} and \eqref{Nbc} provides a formidable mathematical and numerical challenge. It will require substantial new ideas to develop theories for existence, uniqueness, boundedness etc. We hope our work stimulates further research in this direction. 

It should be noted that the MFPT equations derived here have no drift term, i.e. no first order term. This is by design, since we assume symmetry of the underlying kernels, $K(x,v,v')=K(x,v,-v')$ and $q(x,v)=q(x,-v)$,  see Table \ref{tab:assumptions} above. These symmetry conditions ensure that certain integrals in our derivation are zero. If these conditions are relaxed, those integrals stay, and we expect them to lead to drift terms in the end result. We did not work out the details here, since 
our focus was on developing the overall framework to derive the MFPT equation for transport equations. By relaxing the above assumptions, more general models can be derived. 

The bimodal von-Mises distribution is an excellent non-trivial example to study the models developed here. It obeys the conditions 
listed above and allowed us to analytically solve \eqref{anisotropicMFPT3} for several movement types 
and geometries, including radial and circular motion on a disk and on an annulus. We also 
found numerical solutions for a specific preferred orientation (or orientations) as provided 
by the underlying environment. Finally, we compared the MFPTs in these various settings
to the isotropic case and discussed under what conditions anisotropy yields longer
or shorter MFPTs to the boundary of given domains. Our work provides a framework to 
quantify important timescales associated to 
biological or animal movement processes without resolving the entire forward process. 

There are many avenues for future work arising from this study. Our examples of MFPT 
calculation in sections \ref{sec:linear} and \ref{sec:wolf} suggest that the presence of 
multiple linear anisotropic features produces a superposition that is independent of 
individual heterogeneities. Hence, we hypothesize that a homogenization limit can be 
obtained to reflect an effective description of diffusion processes in crowded domains. 
Our principal consideration of the bimodal von-Mises produces unbiased one-dimensional 
diffusion along linear features. It is natural to extend the formulations obtained herein 
towards more general classes of distributions including single modal von-Mises and multi-
modal distributions such as Kent and Bingham distributions in two-dimensions or higher.
These may be more difficult to evaluate as, for example, the symmetry condition $q(x,v) = q(x,-v)$ does not hold for the single modal von-Mises distribution in \eqref{vonMises}. 
Other avenues of exploration include deeper investigations of ecological scenarios, for 
example, by studying how the MFPT changes upon allowing animals to exit only from select
portions of the boundary, modifying the parameter $k_0$ or allowing it to have different 
values along different segments for the purpose of optimally steering herds of animals
to given locations, where, for example, foraging conditions are more suitable. \al{The variance in the MFPT plays an important role in the context of ecological diffusion problems \cite{Venu2015}, hence extensions of this work to the derivation of higher order moments would be valuable. An even more ambitious objective would be to derive the probability density $-s_t(x,v,t)$ from which all moments could be generated.}

\section*{Acknowledgments}
We acknowledge fruitful discussions with Kevin Painter and Stuart Johnson \al{and very helpful comments from two anonymous referees.}
 TH is supported through a discovery grant of the Natural Science and Engineering Research Council of Canada (NSERC), RGPIN-2023-04269. AEL acknowledges support from NSF grant DMS-2052636.
MRD acknowledges support from the Army Research Office through grant W911NF-23-1-0129 and from the National Science Foundation through grant MRI-2320846.

\appendix

\section{Discretization of the anisotropic MFPT equation}
Our numerical simulation of \eqref{anisotropicMFPT} with diffusivity tensor \eqref{radialD} is based on finite difference simulation in \textsc{Matlab}. We write the following for $x=(x_1,x_2)$ and $\nabla = (\partial_{x_1},\partial_{x_2})$ 
\begin{equation}\label{eqn:MainFD}
\mathbb{D}(x) : \nabla\otimes \nabla T = D_{11} \partial^2_{x_1 x_1} T + 2 D_{12}\partial^2_{x_1x_2}T + D_{22} \partial^2_{x_2 x_2}T = -1, \qquad \mathbb{D}(x) = \begin{pmatrix} D_{11} & D_{12}\\D_{12} & D_{22}\end{pmatrix}.
\end{equation}
In the formulation above, the tensor coefficients $D_{11}$, $D_{12}$ and $D_{22}$ are spatially varying functions defined in \eqref{eqn:DiffTensor}. A mesh for the rectangular domain $\Omega = (a,b)\times(c,d)$ is formed at points $x_1^{(j)} = a + (j-1)\Delta x_1$ for $j = 1,\ldots, N_1$ and $x_2^{(k)} = c + (k-1)\Delta x_2$ with $k = 1,\ldots, N_2$, such that 
\[
\frac{b-a}{N_1 -1} = \Delta x_1, \qquad \frac{d-c}{N_2 -1} = \Delta x_2.
\]
We define $T^{(j,k)} = T(x_1^{(j)},x_2^{(k)})$ and utilize the standard finite difference approximations
\bsub\label{app:dis_FD}
\begin{gather}
\partial^2_{x_1x_1}T^{(j,k)} = \frac{T^{(j+1,k)} - 2T^{(j,k)} + T^{(j-1,k)} }{\Delta x_1^2}, \qquad \partial^2_{x_2x_2}T^{(j,k)} = \frac{T^{(j,k+1)} - 2T^{(j,k)} + T^{(j,k-1)} }{\Delta x_2^2},\\[5pt]
\partial^2_{x_1x_2}T^{(j,k)} = \frac{T^{(j+1,k+1)} - T^{(j+1,k-1)} - T^{(j-1,k+1)} + T^{(j-1,k-1)} }{4\Delta x_1 \Delta x_2}.
\end{gather}
\esub
Hence, combining the discretizations \eqref{app:dis_FD} with the equation \eqref{eqn:MainFD}, we arrive at the discrete linear system
\begin{gather*} D_{11}^{(j,k)} \frac{T^{(j+1,k)} - 2T^{(j,k)} + T^{(j-1,k)} }{\Delta x_1^2} +  
D_{12}^{(j,k)} \frac{T^{(j+1,k+1)} - T^{(j+1,k-1)} - T^{(j-1,k+1)} + T^{(j-1,k-1)} }{2\Delta x_1 \Delta x_2}
\\[4pt]
+\ D_{22}^{(j,k)} \frac{T^{(j,k+1)} - 2T^{(j,k)} + T^{(j,k-1)} }{\Delta x_2^2} = -1.
\end{gather*}
In our simulation, we apply the Dirichlet boundary so that $T^{(j,k)}=0$ for any $j \in \{1,N_1\}$ and any $k\in\{1,N_2\}$.

\bibliographystyle{siam}
\bibliography{refs}

\end{document}